\documentclass[reqno,11pt,oneside]{amsart}
\usepackage{amsmath,amssymb,amsthm,mathrsfs}
\usepackage{epsfig,color}
\usepackage[latin1]{inputenc}
\usepackage[english]{babel}
\usepackage[numbers]{natbib}
\usepackage{mathtools}
\usepackage{braket}

\usepackage{ulem}

\usepackage{esint}

\usepackage{hyperref}

\DeclarePairedDelimiter{\norm}{\lVert}{\rVert}
\DeclarePairedDelimiter{\abs}{\lvert}{\rvert}

\newcommand{\rn}{\mathbb{R}^n}

\def\R{\mathbb R}

\def\e{\varepsilon}

\def\vphi{\varphi}

\def\l{\lambda}

\def\00{{\bf 0}}

\newcommand{\<}{\langle}
\renewcommand{\>}{\rangle}
\renewcommand{\a}{\alpha}

\renewcommand{\l}{\lambda}
\renewcommand{\L}{\Lambda}

\newtheorem*{theorem*}{Theorem}
\newtheorem{theorem}{Theorem}[section]
\newtheorem{lemma}[theorem]{Lemma}
\newtheorem{proposition}[theorem]{Proposition}
\newtheorem{example}[theorem]{Example}
\newtheorem{corollary}[theorem]{Corollary}
\newtheorem{remark}[theorem]{Remark}

\title[Regularity results for anisotropic equations]{Interior regularity results for inhomogeneous anisotropic quasilinear equations}

\author{Carlo Alberto Antonini}
\address{C. A. Antonini. Dipartimento di Matematica "Federigo Enriques",
Universit\`a degli Studi di Milano, Via Cesare Saldini 50, 20133 Milano, Italy}
\email{carlo.antonini@unimi.it}

\author{Giulio Ciraolo}
\address{G. Ciraolo. Dipartimento di Matematica "Federigo Enriques",
Universit\`a degli Studi di Milano, Via Cesare Saldini 50, 20133 Milano, Italy}
\email{giulio.ciraolo@unimi.it}

\author{Alberto Farina}
\address{LAMFA, UMR CNRS 7352, Universit\'e Picardie Jules Verne 33, rue St Leu, 
80039 Amiens, France}
\email{alberto.farina@u-picardie.fr}

\begin{document}

\begin{abstract} 
We consider inhomogeneous $p$-Laplace type equations of the form $-\mathrm{div}\left(a(\nabla u)\right)=f$ in a possibly anisotropic setting. Under general assumptions on the source term $f$, we obtain quantitative Sobolev regularity results for the stress field $a(\nabla u)$ and weighted $L^2$ estimates for the Hessian of the solution. As far as we know, our results are new or refine the ones available in literature also when restricted to the Euclidean setting.
\end{abstract}

\maketitle

\section {Introduction}

In this paper we study local regularity of solutions to inhomogeneous nonlinear PDE's driven by anisotropic $p$-Laplace type operators. More precisely, we are considering equations with a variational structure and of $p$-Laplacian type, possibly singular or degenerate. The word anisotropic means that the considered equation is of quasilinear type and that the gradient of the solution is measured in terms of a norm $H$, i.e.,  we are considering equations of the form 
\begin{equation}\label{eqn1}
-\mathrm{div}\left(a(\nabla u)\right)=f \,,
\end{equation}
where 
\begin{equation} \label{a_def1}
a(\nabla u) := \frac{1}{p} \nabla_\xi H^p (\nabla u) 
\end{equation}
and $H$ is a suitable norm.

We were led to discuss this topic while we were studying qualitative properties for quasilinear  partial differential equations of the form
\begin{equation*}
-\Delta_p u = F(u) \,, 
\end{equation*}
as well as for their natural generalization in an anisotropic setting, and we noticed that some regularity results needed in our analysis were missing. More precisely, we needed quantitative higher order integrability properties for the so-called stress field, i.e. the vector field given by \eqref{a_def1},
as well as for the Hessian of the solution $u$. Few of these results were available when $H$ is the standard Euclidean norm and, in the more general anisotropic setting, only when the source term $f$ is constant.

\medskip
Throughout this paper $\Omega$ is an open subset of $\R^n$ with $n \geq 2$ and, for $1< p < + \infty$,  we consider a local weak solution 
$u \in W^{1,p}_{loc}(\Omega)$ to 
\begin{equation}\label{eqn1}
-\mathrm{div}\left(a(\nabla u)\right)=f \,,
\end{equation}
where $f\in L^q_{loc}(\Omega) $, with
\begin{equation}\label{def:q}
q=
\begin{cases}
2\quad&\text{if }p\geq\frac{2n}{n+2}
\\
(p^*)'\quad&\text{if }1<p<\frac{2n}{n+2}.
\end{cases}
\end{equation}
Here \footnote{Given a function $u: \Omega \to \mathbb{R}$, we denote by $\nabla u (x)$ the gradient of $u$ evaluated at a point $x \in \Omega$. Given a function $\psi: \mathbb{R}^n \to \mathbb{R}$, the notation $\nabla_\xi \psi(Du)$ means that we are differentiating the function $\psi$ with respect to $\xi \in \mathbb{R}^n$ and evaluating it at $\nabla u$.}
\begin{equation} \label{a_def}
a(\nabla u) := \frac{1}{p} \nabla_\xi H^p (\nabla u)   \,,
\end{equation}
and we assume that the norm $H$ is of class $C^2(\R^n\setminus \{0\})$ and  such that
\begin{equation}\label{eqn:ellH}
\text{the unit ball } \{H(\xi)<1\} \textmd{ is uniformly convex};\footnote{i.e. such that the principal curvatures of its boundary are bounded away from zero.}
\end{equation}
see Section \ref{sect_notation} for further properties, equivalent definitions and some explicit examples. Equation \eqref{eqn1} has a variational structure since it is the Euler-Lagrange equation of the functional 
\begin{equation*} 
\mathcal{J}(v)=\frac{1}{p}\int_{\Omega}H^p(\nabla v) \,dx-\int_{\Omega} f v\,dx \,.
\end{equation*}
In particular, if $H$ is the standard Euclidean norm then the corresponding operator on the left-hand side in \eqref{eqn1} is the standard $p$-Laplace operator and, as we will discuss later, even in this special case some of our results are new or refine the existing ones.

Our first main result is a local regularity result regarding the stress field $a(\nabla u)$, more precisely we have 

\begin{theorem} \label{thm_main1}
Let $u\in W^{1,p}_{loc}(\Omega)$ be a local weak solution of \eqref{eqn1}, with $f \in L^q_{loc}(\Omega)$ and where $q$ and $H$ satisfy \eqref{def:q} and \eqref{eqn:ellH}, respectively. Then
\[
a(\nabla u)\in H^1_{loc}(\Omega)
\]
and there exists a constant $C$, depending only on $n,p$ and $H$, such that 
\begin{equation}\label{est:nablaanyp}
\norm{\nabla a(\nabla u)}_{L^2(B_{R/2})} \leq C \Bigl[ (R^{-\frac{n}{2}-1}) \norm{a(\nabla u)}_{L^1(B_{2R} \setminus B_R)} + \norm{f}_{L^2(B_{2R} )} \Bigr],
\end{equation}
\begin{equation}\label{est:anyp}
\norm{ a(\nabla u)}_{L^2(B_{R})} \leq C \Bigl[ R^{-\frac{n}{2}} \norm{a(\nabla u)}_{L^1(B_{2R} \setminus B_R)} + R \norm{f}_{L^2(B_{2R} )} \Bigr],
\end{equation} 
\begin{equation}\label{est:anypL1}
\norm{a(\nabla u)}_{L^1(B_{2R} \setminus B_R)} \leq C \norm{\nabla u}_{L^{p-1}(B_{2R}  \setminus B_R)}^{p-1} \,,
\end{equation} 
for any open ball $B_{2R} \subset \subset \Omega$. 
\end{theorem}

Motivated by applications to qualitative studies of PDEs (as discussed before), we also prove some regularity results regarding the Hessian of the solutions, provided that the source term $f$ enjoys better integrability properties.

\begin{theorem} \label{thm_D2u}
Assume $1<p\leq 2$ and let $u\in W^{1,p}_{loc}(\Omega)$ be a local weak solution of \eqref{eqn1} where $H$ satisfies \eqref{eqn:ellH} and $f \in L^r_{loc}(\Omega)$, $r>n$.
Then
\[
u\in H^2_{loc}(\Omega) \cap C^{1,\beta}_{loc}(\Omega)
\]
for some $\beta \in (0,1)$ depending only on $n,p, r$ and $H$. 

Moreover, for any open ball $B_{2R} \subset \subset \Omega$ we have 
\[
\int_{B_{R/2}}\norm{D^2 u}^2 dx\leq C \Bigl[ R^{-n-2} \norm{a(\nabla u)}_{L^1(B_{2R} \setminus B_R)}^2 + \norm{f}_{L^2(B_{2R} )}^2 \Bigr],
\]
where  $C$ is a constant depending only on $p,n,H, r, B_R, B_{2R}, \norm{u}_{W^{1,p}(B_{2R})}, \norm{f}_{L^r(B_{2R})}.$

In particular, when $p=2$ we have 
\[
\int_{B_{R/2}}\norm{D^2 u}^2 dx\leq C \Bigl[ R^{-n-2} \norm{a(\nabla u)}_{L^1(B_{2R} \setminus B_R)}^2 + \norm{f}_{L^2(B_{2R} )}^2 \Bigr],
\]
where  $C$ is a constant depending only on $n,H$.
\end{theorem}

\begin{remark} \label{rem-thm_D2u}
Theorem \ref{thm_D2u} is a special case of a more general result involving a source term $f$ satisfying some weaker integrability conditions. See Theorem  \ref{thm_D2u-generale} and Remark \ref{caso-particolare} in Section \ref{sect_proofs}.
\end{remark}

For a general $p>1$ we have the following weighted integral estimate for the Hessian of the solution $u$. 

\begin{theorem} \label{thm_main2}
Let $u\in W^{1,p}_{loc}(\Omega)$ be a local solution of \eqref{eqn1}, where $H$ satisfies \eqref{eqn:ellH} and $f\in L^r_{loc}(\Omega)$ with $r>n$. Then
\[
u\in H^2_{loc}(\Omega \setminus Z) \cap C^{1,\beta}_{loc}(\Omega)
\]
where $Z$ denotes the set of critical points of $u$ and $\beta \in (0,1)$ depends only on $n,p,r$ and $H$.

\smallskip

Moreover, for any open ball $B_{2R} \subset \subset \Omega$ we have 
\begin{equation}\label{int:est}
\int_{B_{R/2} \setminus Z}\left[H^2(\nabla u)\right]^{p-2}\norm{D^2 u}^2 dx\leq C,
\end{equation}
where $C$ is a constant depending only on $p,n,H, r, B_R, B_{2R}, \norm{u}_{W^{1,p}(B_{2R})}, \norm{f}_{L^r(B_{2R})}.$
\end{theorem}

\begin{remark}\label{rem-thm_main2}
	Theorem \ref{thm_main2} is a special case of a more general result involving a source term $f$ satisfying some weaker integrability conditions. See Theorem  \ref{thm_main2-generale} and Remark \ref{caso-particolare} in Section \ref{sect_proofs}.
\end{remark}


Let us now briefly overview the results related to ours and which are available in the existing literature. 
A first result concerning the local Sobolev regularity of the stress field was proved in \cite{lou} for the special case of the classical $p$-Laplacian operator.\footnote{Actually in \cite{lou} the author proves only that $ \vert \nabla u \vert^{p-1} \in H^1_{loc}(\Omega, \R)$ and not that $ \vert \nabla u  \vert^{p-2} \nabla u  \in H^1_{loc}(\Omega, \R^N)$.} Also, the results in  \cite{lou} are obtained under stronger (than ours) integrability assumptions on the source term\footnote{See the discussion after formula \eqref{def-q}.}  and the quantitative estimates are not obtained. 

In \cite{CianchiMazya} an equation driven by a rotationally-invariant operator is considered, i.e. an equation having the special form 
\begin{equation}\label{cm1}
-\mathrm{div}\left (\mathfrak a(\vert \nabla u \vert) \nabla u \right)=f
\end{equation}
where $\vert \cdot \vert$ denotes the Euclidean norm. 
	Under the so-called Uhlenbeck structure conditions, the authors of \cite{CianchiMazya} prove local\footnote{In \cite{CianchiMazya} the authors, under suitable regularity assumptions on the bounded domain 
	$\Omega $, also prove global (i.e., up to the boundary) $H^1$-regularity for the stress field $\mathfrak a(\vert \nabla u \vert) \nabla u$ when $u$ is a solution to either the homogeneous Dirichlet or the homogeneous Neumann problem.} $H^1$-regularity for  the stress field
	$\mathfrak a(\vert \nabla u \vert) \nabla u$ together with a quantitative estimate. 
	Their approach is different from ours. They make use of an intermediate inequality for the square of the differential operator $ -\mathrm{div}\left (\mathfrak a(\vert \nabla u \vert) \nabla u \right)$. However this differential inequality seems to depend crucially on the fact that the left-hand-side of \eqref{cm1} is rotationally invariant and therefore its anisotropic counterpart does not seem to be obvious. 
	
Fractional-Sobolev regularity for the stress fields has also been investigated. 
In particular, the authors of \cite{akm} prove that the stress field $a(\nabla u)$ belongs to $W^{\sigma,1}_{loc}$ for any $\sigma \in (0,1)$, whenever $f$ is locally integrable (or even a suitable measure). 
	
While writing this paper, we became aware of [11, Theorem 1.2] where the authors obtain some quantitative estimates (in $H^1_{loc}$) on the stress field $a(\nabla u)$ for a larger class of operators than ours. 
Their approach completely differs from ours, it provides some estimates in a slightly different form and it seems that it does not lead to integral estimates for the Hessian of the solution $u.$

Second-order Sobolev regularity for solutions to the inhomogeneous $p$-Laplace equation has also been the object of research. For $p \in (1,2]$ and $ f \in L^{p'}$, the regularity $u \in W^{2,p} $ has been obtained in \cite{DeTh} and \cite{Sim}.
To the best of our knowledge, when $p>2$, the known results are available only under a Sobolev-type regularity for the source term $f$. Indeed, the author of \cite{Cel} proves that $u \in H^2_{loc}$ if $ f \in H^1_{loc}$, when $p \in (2,3)$, while in \cite{MRS} the regularity $u \in W^{2,m}_{loc} $ if $ f \in W^{1,m}_{loc}$, $m>n$, is obtained when $p$ is suitably close to $2$. We also refer to \cite{DS} for an earlier contribution under stronger regularity assumptions on both $u$ and $f$. Finally we mention that, for $p>2$ and $ f \in L^r$,  fractional-Sobolev regularity results for the gradient of solutions to nonlinear equations of $p$-Laplacian type can be found in \cite{Sim}, \cite{Ming}, \cite{akm} (see also the references therein).

Apart from its own interest in regularity theory, Theorem \ref{thm_main1} may be helpful in many situations arising in PDEs theory and we actually arrived to study this problem while we were working on another project on qualitative properties of solutions to elliptic PDEs which will appear in a forthcoming paper. However, in this paper we also prove two interesting consequences of Theorem \ref{thm_main1}.

The first application is related to the measure of critical points and it was firstly proved in \cite{lou} in the Euclidean case and under more restrictive assumptions on $f$ (see also \cite{ciraolino}).

\begin{proposition}\label{prop:1}
Let $u\in W^{1,p}(\Omega)$ be a weak solution of \eqref{eqn1} and assume that the assumptions of Theorem \ref{thm_main1} are fulfilled. 
Then
\[
f(x)=0\quad\text{a.e. } x\in\{\nabla u=0\}.
\]
\end{proposition}

An immediate consequence of Proposition \ref{prop:1} is the following corollary.

\begin{corollary} \label{corollary_final}
Under the assumptions of Proposition \ref{prop:1}, if $f(x)\neq 0$ for almost all $x\in \Omega$, then the Lebesgue measure of the singular set $\{\nabla u=0\}$ is zero. 

In particular, for any $C\in\R$, the level set $\{ u=C\}$ has zero measure.
\end{corollary}

\medskip

The paper is organized as follows. In Section \ref{sect_notation} we introduce some notation, clarify the setting in which we are working and provide some examples of norms satisfying \eqref{eqn:ellH}. In Section \ref{sect_approximation} we describe our approximation argument and obtain some preliminary uniform bounds. Section \ref{sect_unif_bounds} is devoted to the proof of some crucial uniform bounds for the approximating solutions. The proofs of the main results are given in Section \ref{sect_proofs}.

\medskip 

\section*{Acknowledgements}
The first two authors have been partially supported by the ``Gruppo Nazionale per l'Analisi Matematica, la Probabilit\`a e le loro Applicazioni'' (GNAMPA) of the ``Istituto Nazionale di Alta Matematica'' (INdAM, Italy).


\section {Notations and the anisotropic setting} \label{sect_notation}
In this section we clarify the notation, make some comments on the main assumptions and provide examples of anisotropic norms.

Let $\Omega\subseteq \R^n$, $n \geq 2$, and let $p \in(1, \infty)$. Given a function $u: \Omega \to \mathbb{R}$, we denote by $\nabla u (x)$ the gradient of $u$ evaluated at a point $x \in \Omega$. Given a function $\psi: \mathbb{R}^n \to \mathbb{R}$, the notation $\nabla_\xi \psi(Du)$ means that we are differentiating the function $\psi$ with respect to $\xi \in \mathbb{R}^n$ and evaluating it at $\nabla u$.

Let  $H: \mathbb{R}^n \to \mathbb{R}$ be a norm of $\mathbb{R}^n$. Throughout the paper, we assume that $H$ is of class $C^2(\R^n\setminus \{0\})$ 
and we ask that its anisotropic unit ball 
\begin{equation}\label{def-ellH}
B_1^H=\{\xi\in\R^n\,:\,H(\xi)<1  \} \, \quad \text{is \textit{uniformly} convex.}
\end{equation}
This means that all the principal curvatures of its boundary are bounded away from zero (see for instance \cite{cfv}).


In view of the smoothness assumptions on the norm $H$ we have 
$$
\frac{H^p}{p} \in C^1(\R^n) \cap C^2(\R^n\setminus \{0\})
$$ 
and we shall denote its gradient by $a=a(\xi)$,  i.e.,

\begin{equation}\label{def-a}
a=a(\xi) =
\begin{cases}
H^{p-1}(\xi)\nabla_\xi H(\xi) &\text{if } \xi\neq 0,
\\
0 &\text{if } \xi = 0.
\end{cases}
\end{equation}


We consider a local weak solution $u \in W^{1,p}_{loc}(\Omega)$ to 
\begin{equation}\label{def-eqn-studiata}
-\mathrm{div}\left(a(\nabla u)\right)=f \quad\text{in }\Omega,
\end{equation}
i.e., a function $u \in W^{1,p}_{loc}(\Omega)$ satisfying 
\begin{equation}\label{def-eqn-studiata-forma-integrale}
\int_{\Omega}a(\nabla u) \nabla \varphi \,dx = \int_{\Omega} f \varphi \,dx \qquad \forall \, \varphi \in W^{1,p}_c(\Omega), 
\end{equation} 
where $W^{1,p}_{c}(\Omega)$ denotes the the set of compactly supported members of $W^{1,p}(\Omega)$ and the source term $f$ is assumed to belong to $L^q_{loc}(\Omega)$  with
\begin{equation}\label{def-q}
q=
\begin{cases}
2\quad&\text{if }p\geq\frac{2n}{n+2} \,,
\\
(p^*)'\quad&\text{if }1<p<\frac{2n}{n+2} \,.
\end{cases}
\end{equation}
Let us remark that the assumption $q= (p^*)'$, when $1<p<\frac{2n}{n+2}$, is the least one on the source term $f$ in order to have the right-hand side of equation \eqref{def-eqn-studiata-forma-integrale} well defined.\footnote{For $1<p<\frac{2n}{n+2}$, one might use a notion of very weak (or generalized) solution. However, for the sake of clarity and to avoid burying main ideas under technical details, we will not consider this context.}

Also, for $1<p<\frac{2n}{n+2}$, we always have $2= \left(\frac{2n}{n+2}\right)^* = \left(\left(\frac{2n}{n+2}\right)^* \right)' < (p^*)' = \frac{np}{np-n+p} <\frac{n}{p}$. Therefore, our integrability assumption on $f$ is weaker than the one in \cite{lou}. \footnote{Since \eqref{def-a} and \eqref{def-q} are satisfied, a standard density argument implies that any distributional solution $u \in W^{1,p}_{loc}(\Omega)$ of  \eqref{def-eqn-studiata} is a local weak solution.}


The assumption \eqref{def-ellH} on the anisotropic unit ball of $H$ implies 
that $a=a(\xi)$ satisfies some natural growth and ellipticity conditions.  Indeed we notice that, by letting
\begin{equation} \label{B_def}
B(t)=\frac{t^p}{p}\quad\text{for }t>0,
\end{equation}
equation \eqref{def-eqn-studiata} can be written as 
\begin{equation}
-\mathrm{div}\left(\nabla_\xi (B\circ H)(\nabla u)  \right)=f\quad\text{in }\Omega.
\end{equation}
Then, according to \cite[Proposition 3.1]{cfv}, there exist constants $c,C>0$, depending only on $n,p, H$, such that
\begin{equation}\label{eqn:ellHp}
\begin{split}
&\partial_{\xi_i\xi_j}(B\circ H)(\xi)\eta_i\eta_j\geq c \abs{\xi}^{p-2}\abs{\eta}^2
\\
&\sum_{i,j=1}^n\left|\partial_{\xi_i\xi_j}(B\circ H)(\xi)\right|\leq C \abs{\xi}^{p-2},
\end{split}
\end{equation}
for all $\xi\in\R^n\setminus\{0\}$, $\eta\in\R^n$.


For $n \geq 2$ we shall 
denote by $C_s(r,n)$ the Sobolev constant of the embedding $W^{1,r}\hookrightarrow L^{r^*}$ in $\rn$, for $1<r<n$. We also recall that, when
$p\geq 2n/(n+2)$, by Sobolev and Holder inequality we get
\begin{equation}\label{dis2}
\begin{split}
\norm{v}_{L^{q'}(\omega)}&=\norm{v}_{L^2(\omega)}\leq C_s\left(\frac{2n}{n+2},n\right)\norm{\nabla v}_{L^{2n/(n+2)}(\omega')}
\\
&\leq C_s\left(\frac{2n}{n+2},n\right)\abs{\omega}^{\frac{1}{2}+\frac{1}{n}-\frac{1}{p}}\norm{\nabla v}_{L^p(\omega)}\qquad \qquad\forall \, v\in W^{1,p}_0(\omega)
\end{split}
\end{equation} 
where $ \omega$ is any open bounded subset of $\R^N$.


%

Finally we recall that the dual norm of $H$, which is denoted by $H_0$, is defined by
$$
H_0(x)=\sup_{\xi\neq 0}\frac{x \cdot \xi}{H(\xi)}
\qquad \forall \, x \in\mathbb{R}^n 
$$
and satisfies the following property (see for instance \cite[Lemma  3.1]{CS})

\begin{equation}\label{prop-H_0} 
H_0(\nabla_\xi H(\xi))=1 \qquad \forall \, \xi \in  \R^n\setminus \{0\}.
\end{equation}

\medskip

We conclude this section by mentioning that interesting examples of norms satisfying \eqref{def-ellH} can be found in \cite{cfv}. In the following example we provide a further one.

\begin{example} 
Let $H_\sharp$ and $H_*$ be two norms of class $C^2(\R^n\setminus \{0\})$, and assume that $H_\sharp$ satisfies \eqref{def-ellH} (and therefore also  \eqref{eqn:ellHp}). Let $a,b >0$ and define 
$$
H(\xi)=\left(aH_\sharp^p (\xi)+bH_*^p (\xi) \right)^{1/p}.
$$
Then $H$ satisfies \eqref{eqn:ellHp} which, in view of \cite[Proposition 3.1]{cfv}, is equivalent to say that $H$ satisfies  \eqref{def-ellH}. 

Indeed, it is clear that $H$ is a norm and that 
$$
B\circ H(\xi) := \frac{H^p(\xi)}{p} = \frac{a}{p}H_\sharp^p (\xi)+\frac{b}{p}H_*^p (\xi) \,.
$$
Since $H_*$ is one-homogeneous and of class $C^2$ outside the origin, we have that $\nabla_\xi^2 H_*^p$ is homogeneous of degree $p-2$, and so the second inequality in \eqref{eqn:ellHp} is fulfilled. The first condition in \eqref{eqn:ellHp} follows from the fact that $H_\sharp$ satisfies \eqref{eqn:ellHp} and since $H_*^p$ is convex and hence its Hessian is nonnegative definite outside the origin. 

As a particular case of this example, we notice that the assumptions on $H_\sharp$ are clearly satisfied by the Euclidean norm $|\cdot|$.

\end{example}


\section {The approximation argument} \label{sect_approximation}
As usual in regularity theory, the starting point of our argument is the choice of an approximating procedure. In this section we set the approximation argument and obtain a preliminary uniform bound which will be useful later.

Let $\e \in [0,1)$ and set $B_\e(t)=B(\sqrt{\e^2+t^2})-B(\e)$ with $B$ given by \eqref{B_def}, i.e.
$$
B_\e(t) =\frac{1}{p}\left(\e^2+t^2  \right)^{\frac{p}{2}}-\e^p/p\
$$
for any $t \geq 0$. 

We set $f_0 := f$ and 
\begin{equation}\label{def:f-eps}
f_\e :=  \min \Big \{\max \{f, -{\e^{-1} \}, {\e^{-1}} } \Big \}  \qquad  \qquad \forall  \, \e \in (0,1)\, ;
\end{equation}
then
\begin{equation}\label{prop:f-eps}
\begin{cases}
& f_\e \in  L^\infty(\Omega),  \quad \vert f_\e \vert \leq \vert f \vert \quad \text{ a.e. in } \Omega, \\
&f_\e\to f\quad\text{in } L^q_{loc}(\Omega).
\end{cases}
\end{equation} 

Let us fix a subdomain $\Omega' \subset\subset \Omega$ (i.e. compactly contained in $ \Omega$) and let $u_\e$ be the unique weak solution of
\begin{equation}\label{eqn:ue}
\begin{cases}
-\mathrm{div}\left(\nabla_\xi (B_\e\circ H)(\nabla u_\e)  \right)=f_\e\quad\text{in }\Omega'
\\
u_\e=u \quad\text{on }\partial\Omega',
\end{cases}
\end{equation}
where the boundary condition is to be intended as
\[
u_\e-u\in W^{1,p}_0(\Omega').
\]
It is classical that, for every $\e\in [0,1)$, $u_\e$ is the unique minimizer of the strictly convex, coercive and weakly lower semicontinuous functional 
\begin{equation}\label{def:fune}
\mathcal{J}_\e(v)=\frac{1}{p}\int_{\Omega'}\left(\e^2+H^2(\nabla v)  \right)^{\frac{p}{2}}\,dx-\int_{\Omega'} f_\e v\,dx,
\end{equation}
in the closed and convex set
\[
W^{1,p}_u(\Omega')=u+W^{1,p}_0(\Omega').
\]

Now, thanks to \cite[Proposition 3.1 and Lemma 4.1]{cfv}, there exist constants $c, C>0$, depending only on $n,p, H$, such that
\begin{equation}\label{ell:eps}
\partial_{\xi_i\xi_j}(B_\e\circ H)(\xi)\eta_i\eta_j\geq c(\e^2+\abs{\xi}^2)^{\frac{p-2}{2}}\abs{\eta}^2
\end{equation}
and
\begin{equation}\label{ell:eps2}
\sum_{i,j=1}^n\left|\partial_{\xi_i\xi_j}(B_\e\circ H)(\xi)\right|\leq C(\e^2+\abs{\xi}^2)^{\frac{p-2}{2}},
\end{equation}
for all $\eta\in \R^n$, $\xi\in\R^n\setminus \{0\}$. Therefore there exist positive constants $\l,\L$, depending only on $n,p,H$ such that, for all $\eta\in \R^n$, $\xi\in\R^n\setminus \{0\}$, it holds
\begin{equation}\label{ell:fin}
\l \abs{\eta}^2\leq \frac{\<D^2_\xi (B_\e\circ H)(\xi)\eta,\eta\>}{[\e^2+\abs{\xi}^2]^{\frac{p-2}{2}}}\leq \L\abs{\eta}^2.
\end{equation}

The following lemma provides a first useful bound on the approximating functions $u_\e$.

\begin{lemma}
Let $u_\e$ be a solution of \eqref{eqn:ue}. Then, for any $\Omega' \subset\subset \Omega$ and for any $\e\in (0,1)$,
\begin{equation}\label{eqn:stime1}
\int_{\Omega'}\left(\e^2+H^2(\nabla u_\e)\right)^{\frac{p}{2}}dx\leq K_{\Omega'} + 2^p \e^p \abs{\Omega'}
\end{equation}
with
\begin{equation} \label{K_def}
K_{\Omega'} = (2^p+1)
\int_{\Omega'}H^p(\nabla u)
+\underline{C}\,\norm{f}_{L^q(\Omega')}^{p'} .
\end{equation}
Here $ \abs{\Omega'}$ denotes the Lebesgue measure of $\Omega'$ and 
$\underline{C}=\underline{C}(n,p,H,\abs{\Omega'})$ is a non-negative constant, independent of  $\e$, that can be explicitly determined. \footnote{ \, \, $\underline{C} = 2^{p'+1}(p-1)\a^{p'}C_0^{p'}$,  where $C_0$ is given by \eqref{valore-C_0} and $\a>0$ is a structural constant such that $\abs{\xi}\leq \a\,H(\xi)$ for any $\xi \in \R^N$. Note that such a constant $\a$ does exist since all the norms on $\R^N$ are equivalent.} 

Furthermore, we have that
\[
u_\e \longrightarrow u \quad\text{strongly in }W^{1,p}(\Omega').
\]
\end{lemma}


\begin{proof}
Since $u_\e$ minimizes the functional \eqref{def:fune} over $W^{1,p}_u(\Omega')=u+W^{1,p}_0(\Omega')$, we can take $u$ as a competitor. This choice leads to  

\begin{equation}\label{ineq:weak}
\begin{split}
\frac{1}{p}\int_{\Omega'} \left(\e^2+H^2(\nabla u_\e)  \right)^{\frac{p}{2}}\,dx &\leq \frac{1}{p}\int_{\Omega'}\left(\e^2+H^2(\nabla u)  \right)^{\frac{p}{2}}\,dx+\int_{\Omega'} f_\e(u_\e-u)\,dx
\\
&\leq \frac{1}{p}\int_{\Omega'}\left(\e^2 +H^2(\nabla u)  \right)^{\frac{p}{2}}\,dx+\norm{f_\e}_{L^q({\Omega'})}\norm{u_\e-u}_{L^{q'}(\Omega')}.
\end{split}
\end{equation}
Then, 
\begin{equation}
\begin{split}
\norm{u_\e-u}_{L^{q'}(\Omega')}&=
\begin{cases}
\norm{u_\e-u}_{L^2}\quad&\text{if } p\geq 2n/(n+2)
\\
\\
\norm{u_\e-u}_{L^{p^*}}\quad&\text{if } p< 2n/(n+2)
\end{cases}
\\
\\
&\leq
\begin{cases}
C_s\left(\frac{2n}{n+2},n\right) \norm{\nabla u_\e-\nabla u}_p\,\abs{\Omega'}^{\frac{1}{2}+\frac{1}{n}-\frac{1}{p}}\quad&\text{if } p\geq 2n/(n+2)
\\
\\
C_s(p,n) \norm{\nabla u_\e-\nabla u}_p \quad&\text{if } p<2n/(n+2)
\end{cases}
\end{split}
\end{equation}
where in the latter we have used  \eqref{dis2}.
Hence, 
\begin{equation}
\norm{u_\e-u}_{L^{q'}(\Omega')} \leq  C_0 \Big(\norm{\nabla u_\e}_{L^p(\Omega')}+\norm{\nabla u}_{L^p(\Omega')} \Big)
\end{equation}
where 
\begin{equation}\label{valore-C_0}
C_0=
\begin{cases}
C_s\left(\frac{2n}{n+2},n\right) \,\abs{\Omega'}^{\frac{1}{2}+\frac{1}{n}-\frac{1}{p}}\quad&\text{if } p\geq 2n/(n+2)
\\
\\
C_s(p,n) \quad&\text{if } p<2n/(n+2) \,.
\end{cases}
\end{equation}

Therefore, for any $\delta>0$, by weighted Young's inequality we obtain
\begin{equation*}
\begin{split}
\norm{f_\e}_{L^q(\Omega')}\,&\norm{u_\e -u}_{L^{q'}(\Omega')}\leq \Big(\norm{\nabla u_\e}_{L^p(\Omega')}+\norm{\nabla u}_{L^p(\Omega')} \Big)C_0 \norm{f_\e}_{L^q(\Omega')}
\\
&\leq \frac{\delta^p}{p}\Big(\norm{\nabla u_\e}_{L^p(\Omega')}+\norm{\nabla u}_{L^p(\Omega')} \Big)^p+\frac{\left( C_0\norm{f_\e}_{L^q(\Omega')}\right)^{p'}}{\delta^{p'}\,p'} \,.
\end{split}
\end{equation*}

By plugging the above inequality in \eqref{ineq:weak} we infer
\begin{equation}
\begin{split}
\frac{1}{p}\int_{\Omega'} &\left(\e^2+H^2(\nabla u_\e)  \right)^{\frac{p}{2}}\,dx\leq\frac{1}{p}\int_{\Omega'}\left(\e^2+H^2(\nabla u)  \right)^{\frac{p}{2}}\,dx+
\\
&+\frac{2^{p-1}\delta^p\a^p}{p}\left(\int_{\Omega'}H^p(\nabla u_\e) dx+ \int_{\Omega'}H^p(\nabla u) dx \right)+\frac{C_0^{p'}\norm{f_\e}_{L^q(\Omega')}^{p'}}{\delta^{p'}p'},
\end{split}
\end{equation}
where $\a>0$ is such that $\abs{\xi}\leq \a\,H(\xi)$. By choosing $\delta=1/2\a$ we find
\begin{equation}
\begin{split}
\int_{\Omega'} \left(\e^2+H^2(\nabla u_\e)  \right)^{\frac{p}{2}}\,&dx\leq 2 \int_{\Omega'}\left(\e^2+H^2(\nabla u)  \right)^{\frac{p}{2}}\,dx+
\\
+&\int_{\Omega'}H^p(\nabla u) dx+2^{p'+1}(p-1)\a^{p'}C_0^{p'}\norm{f_\e}_{L^q(\Omega')}^{p'}
\\
&\leq (2^p+1)\int_{\Omega'}H^p(\nabla u) dx+2^{p'+1}(p-1)\a^{p'}C_0^{p'}\norm{f_\e}_{L^q(\Omega')}^{p'} +  2^p \e^p\abs{\Omega'}
\end{split} 
\end{equation}
and the desired inequality \eqref{eqn:stime1} follows by recalling \eqref{prop:f-eps}.

Now we show that 
\begin{equation*}
u_\e\to u\quad\text{in }W^{1,p}(\Omega').
\end{equation*}
We first notice that $\Vert u_\e \Vert_{W^{1,p}(\Omega')}$ is uniformly bounded in $\e$ 
thanks to Poincar\'e inequality on $\Omega'$ and \eqref{eqn:stime1}. 
We can therefore extract a subsequence, relabeled as $u_\e$, such that
\[
u_\e\rightharpoonup w\quad\text{weakly in }W^{1,p}(\Omega'),
\]
for some function $w\in W^{1,p}_u(\Omega')$, since this set is weakly closed (being closed and convex). 
We want to show that $w=u$ on $\Omega'$.

We recall that $u$ is the unique minimizer of the functional
\begin{equation*}
\mathcal{J}[v]\coloneqq \frac{1}{p}\int_{\Omega'} H^p(\nabla v)\,dx-\int_{\Omega'} f\,v\,dx\quad\text{in } W^{1,p}_u(\Omega').
\end{equation*}
Again, since $\mathcal{J}_\e[u_\e]\leq \mathcal{J}_\e[u] $, we obtain
\begin{equation}\label{ff:1}
\int_{\Omega'}\frac{H^p(\nabla u_\e)}{p}\,dx\leq\frac{1}{p}\int_{\Omega'} \left(\e^2+H^2(\nabla u_\e)\right)^{\frac{p}{2}}dx\leq \frac{1}{p}\int_{\Omega'}\left(\e^2+H^2(\nabla u)  \right)^{\frac{p}{2}}+\int_{\Omega'} f_\e(u_\e-u)\,dx.
\end{equation}

Therefore
\begin{equation}\label{f:1}
\begin{split}
\mathcal{J}[u_\e]&=\frac{1}{p}\int_{\Omega'} H^p(\nabla u_\e)\,dx-\int_{\Omega'} f\,u_\e\,dx
\\
&\leq \frac{1}{p}\int_{\Omega'}\left(\e^2+H^2(\nabla u)  \right)^{\frac{p}{2}}-\int_\Omega f_\e\,u\,dx+\int_{\Omega'} (f_\e-f)\,u_\e\,dx
\\
&=\mathcal{J}_\e[u]+\int_{\Omega'} (f_\e-f)\,u_\e\,dx.
\end{split}
\end{equation}
We know that $f_\e\to f$ in $L^q(\Omega)$ and $u_\e$ is uniformly bounded in $L^{q'}(\Omega')$ by Sobolev inequality; hence
\[
\int_{\Omega'} (f_\e-f)\,u_\e\,dx\to 0\quad\text{as  }\e\to0.
\]
By the weak lower semicontinuity of the functional $\mathcal{J}$ and \eqref{f:1}, we then infer
\begin{equation}
\mathcal{J}[w]\leq \liminf_{\e\to0} \mathcal{J}[u_\e]\leq\liminf_{\e\to0} \left(\mathcal{J}_\e[u]+\int_{\Omega'} (f_\e-f)\,u_\e\,dx\right)=\mathcal{J}[u], 
\end{equation}
which implies that $w=u$ on $\Omega'$ by the uniqueness of minimizers of $\mathcal{J}$. 
By repeating the above argument for any subsequence $\{u_{\e_n}\}\subset\{u_\e\}$, we infer that the whole sequence $u_\e\to u$ weakly in $W^{1,p}(\Omega')$.

\medskip


We now show that $u_\e\to u $ strongly in $W^{1,p}(\Omega)$. By \cite[Lemma 1]{tol}, we have

\begin{equation}
\left[a_\e(\nabla u)-a_\e(\nabla u_\e)\right]\cdot[\nabla u-\nabla u_\e]\geq G_\e\coloneqq \gamma_0
\begin{cases}
(1+\abs{\nabla u}+\abs{\nabla u_\e})^{p-2}\abs{\nabla u-\nabla u_\e}^2\quad &p<2
\\
\abs{\nabla u-\nabla u_\e}^p\quad &p\geq 2,
\end{cases}
\end{equation}
where we set
\begin{equation}\label{def-a-epsilon}
a_\e(\xi)\coloneqq \nabla_\xi (B_\e\circ H)(\xi) =
\begin{cases}
[\e^2+H^2(\xi)]^{\frac{p-2}{2}}H(\xi) \nabla H(\xi) 
&\text{if } \xi\neq 0,
\\
0 &\text{if } \xi = 0.
\end{cases}
\end{equation}
Notice that, by using this notation, $u_\e$ is a weak solution to
\[
-\mathrm{div}(a_\e(\nabla u_\e))=f_\e\quad\text{in }\Omega'.
\]
Therefore
\begin{equation}
\begin{split}
0\leq\int_{\Omega'}  G_\e\,dx&\leq \int_{\Omega'} \left[a_\e(\nabla u)-a_\e(\nabla u_\e)\right]\cdot[\nabla u-\nabla u_\e]\,dx
\\&=\int_{\Omega'}  a_\e(\nabla u)\cdot[\nabla u-\nabla u_\e]\,dx-\int_{\Omega'}  a_\e(\nabla u_\e)\cdot[\nabla u-\nabla u_\e]\,dx
\\
&=I_1(\e)+I_2(\e).
\end{split}
\end{equation}

Now we show that $I_1(\e)$ and $I_2(\e)$ vanish at the limit $\e \to 0$. To this end, we observe  that \eqref{def-a-epsilon} implies 
\[
\abs{a_\e(\nabla u)}  \leq C(H,p) \left(1+\abs{\nabla u}  \right)^{p-1} \quad \text{a.e. in } \Omega', \quad \forall \e \in (0,1),
\]
and so $a_\e(\nabla u)\to a(\nabla u)$ in $L^{p'}(\Omega')$,
by dominated convergence. Since $\nabla u_\e\rightharpoonup \nabla u$ in $L^p(\Omega')$
we immediately obtain that
\[
I_1(\e)\to 0 \quad \text{as} \quad \e\to0 \,.
\]
Regarding $I_2(\e)$, we notice that by testing the equation 
\eqref{eqn:ue} with the test function $u-u_\e$, we have
\begin{equation}
I_2(\e)=-\int_{\Omega'}  f_\e\,(u-u_\e)\,dx.
\end{equation}
We first recall that $u_\e\to u$ weakly in $W^{1,p}(\Omega')$. Moreover, as seen before, $u_\e$ is uniformly bounded in $L^{q'}(\Omega')$ w.r.t. $\e$, then, up to a subsequence, $u_{\e_n}\to u$  weakly in $L^{q'}(\Omega')$. Again, by repeating the argument for any subsequence, we find
\[
u_\e\to u\quad\text{weakly in }L^{q'}(\Omega')\quad\text{and }\,f_\e\to f\quad\text{strongly in } L^q(\Omega) \,,
\]
which imply $I_2(\e)\to 0$ as $\e\to 0$.

Thus we have obtained that
\begin{equation}\label{f:2}
\int_{\Omega'} G_\e\,dx\to0\quad\text{as }\e\to0.
\end{equation}
If $p\geq 2$ then this is exactly the strong convergence of $u_\e$ to $u$ in $W^{1,p}(\Omega')$.

When $p<2$, by Holder's inequality we have
\begin{equation*}
\begin{split}
\int_{\Omega'}\abs{\nabla (u_\e-u)}^p dx 
\leq
&\left(\int_{\Omega'}(1+\abs{\nabla u}+\abs{\nabla u_\e})^{p-2}\abs{\nabla (u_\e-u)}^2 dx \right)^{\frac{p}{2}}
\\
&\times\left( \int_{\Omega'}(1+\abs{\nabla u}+\abs{\nabla u_\e})^p dx \right)^{\frac{2-p}{2}},
\end{split}
\end{equation*}
which goes to $0$ as $\e\to 0$. The latter implies the desired conclusion also for $p<2$,  which concludes the proof.

\end{proof}

The following lemma collects some properties for $u_\e$ which will be useful later.

\begin{lemma} \label{lemma_reg_u_ep}
Let $u_\e$ be a solution of \eqref{eqn:ue}. Then,	
\begin{equation*}
u_\e\in H^2_{loc}(\Omega)\cap C^1(\Omega)
\end{equation*}
and
\begin{equation*}
a_\e(\nabla u_\e) = \Big(a^1_\e(\nabla u_\e),..., a^n_\e(\nabla u_\e)\Big) \in H^1_{loc}(\Omega;\rn).
\end{equation*}
Furthermore, for any $j,k=1,..., n, $
\begin{equation}\label{eqn-derivate}
\partial_{x_k} a^j_\e(\nabla u_\e) = \sum_{m=1}^n 
\frac{\partial a^j_\e}{\partial \xi_m} (\nabla u_\e) \frac{\partial}{\partial x_k} \Big(\frac{\partial u_\e}{\partial x_m} \Big) \qquad {a.e.} \quad {in}\quad \Omega, 
\end{equation}
where the products on the r-h-s are to be interpreted as zero whenever their second factor is zero, irrespective of whether $\frac{\partial a^j_\e}{\partial \xi_m} $ is defined.

\end{lemma}
\begin{proof}
Since $f_\e \in L^\infty_{loc}(\Omega)$, thanks to \cite{ser} we have that $u_\e\in C^0(\Omega)$. Then, thanks to  conditions \eqref{ell:eps}-\eqref{ell:eps2}, we may apply \cite[Theorem 1, Proposition 1]{tol} and obtain
\begin{equation}
\begin{split}
&u_\e\in H^{2}_{loc}(\Omega) \cap C^1(\Omega) \quad\text{if }p\geq2
\\
&u_\e\in W^{2,p}_{loc}(\Omega) \cap C^1(\Omega) \quad\text{if }p\leq2.
\end{split}
\end{equation}
Since $\nabla u_\e \in C^0(\Omega) \subset L^\infty_{loc}(\Omega) $, we infer
\[
u_\e\in H^{2}_{loc}(\Omega)
\]
also in the case $p\leq 2$ by applying \cite[Proposition 4.3]{cfv}.

Now we notice that \cite[Lemma 4.1]{cfv} implies 
\[
a_\e(\xi)\in C^1(\R^n\setminus \{0\})\cap Lip_{loc}(\R^n)
\]
and, from \cite[Theorem 2.1]{mm} (see also \cite[section  11]{leo}), we obtain that 
\[
a_\e(\nabla u_\e)\in H^1_{loc}(\Omega;\rn) 
\]
and \eqref{eqn-derivate}, 
which completes the proof.
\end{proof}


\section {Preliminary uniform bounds}  \label{sect_unif_bounds}
In this section we obtain some crucial integral inequalities for the solutions $u_\e$ of the approximating problems, which allow us to bound some relevant integral quantities uniformly in $\e$.

Let 
\[
Z_\e=\{x\in\Omega \,:\nabla u_\e=0\}
\]
be the set of critical points of $u_\e$. Therefore, in view of Lemma \ref {lemma_reg_u_ep}, 
we have
\[
D^2 u_\e=0\quad\text{a.e. in }Z_\e, 
\] 
and so 
\begin{equation}\label{eq:chain}
\nabla a_\e(\nabla u_\e)=
\begin{cases}
A_\e \, D^2 u_\e \quad& \text{a.e. on }\,Z_\e^c,
\\
0 \quad&\text{a.e. on }\,Z_\e,
\end{cases}
\end{equation}
where the symmetric matrix
\[
A_\e(x)=\nabla_\xi a_\e(\nabla u_\e)
\]
is well defined for $x\not\in Z_\e$.

\begin{proposition} \label{prop_int_bound}
Let $u_\e$ be a solution of \eqref{eqn:ue}. Then there exists a constant $C_1=C_1(n,p,H)$ such that, for any function $\eta\in C_c^{0,1}(\Omega)$ and for any $\e\in (0,1)$, we have 
\begin{equation}\label{eqn:ut1}
\begin{split}
\int_{\Omega} \eta^2[\e^2+H^2(\nabla u_\e)]^{p-2}\norm{D^2 u_\e}^2 dx\leq& C_1\int_{\Omega}  [\e^2+H^2(\nabla u_\e)]^{p-2}H^2(\nabla u_\e)\abs{\nabla \eta}^2\,dx
\\
&+C_1\int_{\Omega}\eta^2 f_\e^2 dx.
\end{split}
\end{equation}
\end{proposition}

\begin{proof}

Since from Lemma \ref{lemma_reg_u_ep} we have that $a_\e(\nabla u_\e) \in H^1_{loc}(\Omega)$, we can differentiate the equation \eqref{eqn:ue} to obtain
\begin{equation}\label{eqn:ued}
-\mathrm{div}\left(\partial_{x_k}a_\e(\nabla u_\e)\right)=\partial_{x_k}f_\e \quad \text{in}  \quad  \mathcal D'( \Omega), \qquad k=1,...,n,
\end{equation}
and so 
\begin{equation}\label{eqn:uedint}
\sum_{j=1}^n \int_{\Omega} \partial_{x_k}a^j_\e(\nabla u_\e) \partial_{x_j}\varphi = - \int_{\Omega} f_\e  \partial_{x_k}\varphi \qquad k=1,...,n,
\end{equation}
holds true for any $ \varphi \in H^1_c(\Omega)$, the set of compactly supported members of  $H^1(\Omega)$.
 
For any $\eta\in C^{0,1}_c(\Omega)$ and any $k=1,...,n$ we first choose $\vphi=\eta^2 a_\e^k(\nabla u_\e)\in H^1_c(\Omega)$ as test function in \eqref{eqn:uedint} and 
then we sum the obtained identities from $k=1$ to $n$ to obtain
\begin{equation}\label{eq1:1}
\begin{split}
0=&\int_{\Omega} \eta^2 \mathrm{tr}\left[\left(\nabla  a_\e(\nabla u_\e)\right)^2\right]\,dx+2\int_{\Omega}\eta \<\nabla a_\e(\nabla u_\e)\,a_\e(\nabla u_\e),\nabla \eta \> dx
\\
&+ \sum_{k=1}^n \int_{\Omega} \eta^2 \partial_{x_k}a_\e^k(\nabla u_\e)f_\e\,dx+2\int_{\Omega} \eta \,f_\e \<a_\e(\nabla u_\e), \nabla \eta\>\,dx=
\\
&=I_1+I_2+I_3+I_4.
\end{split}
\end{equation}
Thanks to the ellipticity condition \eqref{ell:fin}, we have that 
\begin{equation}\label{ell:A}
\l\abs{z}^2\leq \frac{\<A_\e(x)\,z,\,z\>}{[\e^2+H^2(\nabla u_\e(x))]^{\frac{p-2}{2}}}\leq \L\abs{z}^2\quad\forall z\in\R^n,\,\forall x\not\in Z_\e \,,
\end{equation}
where $\l$ and $\L$ depend only on $n,p, H$. 

We exploit the following basic algebra inequality: let $X$ be a symmetric matrix, and $Y$ positive semidefinite matrix; if $\l_{min}$ and $\l_{\max}$ denote the smallest and biggest eigenvalues of $X$, respectively, then
\begin{equation}\label{tr:ineq}
\l_{min}\mathrm{tr}(Y)\leq \mathrm{tr}(XY)=\mathrm{tr}(YX)\leq \l_{max}\mathrm{tr}(Y).
\end{equation}

Therefore, from \eqref{eq:chain}, \eqref{tr:ineq} and  \eqref{ell:A} we infer 


\begin{equation}\label{eqn:I1}
\begin{split}
I_1&=\int_{\Omega\setminus Z_\e} \eta^2 \mathrm{tr}\left[A_\e\, D^2 u_\e\,A_\e\,D^2 u_\e\right]\,dx
\\
&\geq \l\int_{\Omega\setminus Z_\e} \eta^2[\e^2+H^2(\nabla u_\e)]^{\frac{p-2}{2}}\mathrm{tr}\left( D^2 u_\e\,A_\e\,D^2 u_\e \right)\,dx
\\
&=\l\int_{\Omega\setminus Z_\e} \eta^2[\e^2+H^2(\nabla u_\e)]^{\frac{p-2}{2}}\mathrm{tr}\left(A_\e\,D^2 u_\e \,D^2 u_\e\right)\,dx
\\
&\geq\l^2\int_{\Omega\setminus Z_\e} \eta^2[\e^2+H^2(\nabla u_\e)]^{p-2}\mathrm{tr}\left(D^2 u_\e \,D^2 u_\e\right)\,dx
\\
& = \l^2\int_{\Omega} \eta^2[\e^2+H^2(\nabla u_\e)]^{p-2}\norm{D^2 u_\e}^2 dx,
\end{split}
\end{equation}
where we used the symmety of $D^2 u_\e$ and the fact that a scalar product of two matrices $X$ and $Y$ can be defined by $X:Y=\mathrm{tr}(X\,Y^T)$, so that $\norm{X}^2=X:X$.

From \eqref{eq:chain}, \eqref{ell:A} and since 
$$
\vert a_\e(\xi) \vert \leq C(H)[\e^2+H^2(\xi)]^{\frac{p-2}{2}}H(\xi) \qquad \forall \, \xi \in \R^n, \quad \forall \e \in (0,1),
$$
where $C(H)$ is a constant depending only on $H$, we find that 
\begin{equation}\label{eqn:I2}
\begin{split}
&\vert I_2 \vert =  \Big \vert 2\int_{\Omega\setminus Z_\e}\eta \<A_\e\,D^2 u_\e\,a_\e(\nabla u_\e),\nabla \eta \> dx \Big \vert
\\
&\leq 2 \L\,C(H)\int_{\Omega} \eta\,[\e^2+H^2(\nabla u_\e)]^{p-2}H(\nabla u_\e)\abs{\nabla \eta}\,\norm{D^2 u_\e}\,dx
\\
&\leq \delta \int_{\Omega} \eta^2\,[\e^2+H^2(\nabla u_\e)]^{p-2}\norm{D^2 u_\e}^2\,dx
\\
&\quad+\frac{\L^2\,C(H)}{\delta} \int_{\Omega} [\e^2+H^2(\nabla u_\e)]^{p-2}H^2(\nabla u_\e)\abs{\nabla \eta}^2\,dx,
\end{split}
\end{equation}
where in the last inequality we applied the weighted Young inequality with a weight $\delta>0$ to be chosen later.

From \eqref{ell:A}, \eqref{tr:ineq}, Holder and Young inequalities, we obtain
\begin{equation}\label{eqn:I3}
\begin{split}
\vert I_3 \vert &=  \Big \vert \int_{\Omega\setminus Z_\e} \eta^2\,[A_\e:D^2 u_\e]\, f_\e\,dx  \Big \vert
\leq \int_{\Omega\setminus Z_\e} \eta^2\, \norm{A_\e} \norm{D^2 u_\e}\,\abs{f_\e}\,dx
\\
&\leq \sqrt{n} \L\int_{\Omega} \eta^2\,[\e^2+H^2(\nabla u_\e)]^{\frac{p-2}{2}}\norm{D^2 u_\e}\,\abs{f_\e}\,dx
\\
&\leq \delta\int_{\Omega} \eta^2\,[\e^2+H^2(\nabla u_\e)]^{p-2}\norm{D^2 u_\e}^2\,dx+\frac{n\L^2}{4\delta}\int_{\Omega}\eta^2 f_\e^2 dx.
\end{split}
\end{equation}

Finally, via Young's inequality,
\begin{equation}\label{eqn:I4}
\begin{split}
\vert I_4 \vert &\leq 2 C(H)\int_{\Omega}\abs{\eta}\left(\e^2+H^2(\nabla u_\e)\right)^{\frac{p-2}{2}}H(\nabla u_\e)\abs{f_\e}\abs{\nabla \eta}\,dx
\\
&\leq C(H) \int_{\Omega} [\e^2+H^2(\nabla u_\e)]^{p-2}H^2(\nabla u_\e)\abs{\nabla \eta}^2\,dx+\int_{\Omega}\eta^2 f_\e^2 dx.
\end{split}
\end{equation}

By combining \eqref{eqn:I1}-\eqref{eqn:I4} we get
\begin{equation*}
\begin{split}
&(\l^2 - 2\delta) \int_{\Omega} \eta^2[\e^2+H^2(\nabla u_\e)]^{p-2}\norm{D^2 u_\e}^2 dx
\\
& \leq C(H) \left(1+ \frac{\L^2}{\delta} \right) \int_{\Omega} [\e^2+H^2(\nabla u_\e)]^{p-2}H^2(\nabla u_\e)\abs{\nabla \eta}^2\,dx +  \left(1+ \frac{n\L^2}{4\delta} \right) \int_{\Omega}\eta^2 f_\e^2 dx
\end{split}
\end{equation*}
and, by choosing $\delta=\frac{\l^2}{4}$ in the latter,  we find

\begin{equation*}
\begin{split}
&\int_{\Omega} \eta^2[\e^2+H^2(\nabla u_\e)]^{p-2}\norm{D^2 u_\e}^2 dx
\\
& \leq C(H) \frac{2}{\l^2} \left(1+ \frac{4 \L^2}{\l^2} \right) \int_{\Omega} [\e^2+H^2(\nabla u_\e)]^{p-2}H^2(\nabla u_\e)\abs{\nabla \eta}^2\,dx +  \frac{2}{\l^2} \left(1+ \frac{n\L^2}{\l^2} \right) \int_{\Omega}\eta^2 f_\e^2 dx
\end{split}
\end{equation*}
which completes the proof.
\end{proof}

The following corollary is a consequence of Proposition \ref{prop_int_bound}. It will be crucial in the proof of Theorem \ref{thm_main1}. 

\begin{corollary}\label{cor:quasi-stima-finale} Let $u_\e$ be a solution of \eqref{eqn:ue}. Then 
for any function $\eta\in C_c^{0,1}(\Omega)$ and for any $\e\in (0,1)$, we have 
\begin{equation}\label{eqn:ut4-pippo}
\int_{\Omega} \eta^2[\e^2+H^2(\nabla u_\e)]^{p-2}\norm{D^2 u_\e}^2 dx \leq
C_2\int_{\Omega}\abs{a_\e(\nabla u_\e)}^2\abs{\nabla \eta}^2 dx+C_2 \int_{\Omega}\eta^2 f_\e^2 dx
\end{equation}
and
\begin{equation}\label{eqn:ut4}
\int_{\Omega}\eta^2 \norm{\nabla a_\e(\nabla u_\e)}^2\leq C_2\int_{\Omega}\abs{a_\e(\nabla u_\e)}^2\abs{\nabla \eta}^2 dx+C_2 \int_{\Omega}\eta^2 f_\e^2 dx \,,
\end{equation}
where $C_2$ is a constant depending only on $n,p$ and $H$.	
\end{corollary}

\begin{proof} We notice that 
\begin{equation}\label{est:lowerae}
\abs{a_\e(\nabla u_\e)}\geq c(H)\left[\e^2+H^2(\nabla u_\e) \right]^{\frac{p-2}{2}}H(\nabla u_\e).
\end{equation}
Indeed, recalling that the dual norm of $H$ satisfies \eqref{prop-H_0}, we have for any $ \xi \not =0$ 
\begin{equation*}
\begin{split}
H_0(a_\e(\xi))&=H_0\left(\left[\e^2+H^2(\xi)  \right]^{\frac{p-2}{2}}H(\xi)\nabla_\xi H(\xi)\right) \\
&=\left(\left[\e^2+H^2(\xi)  \right]^{\frac{p-2}{2}}H(\xi) \right) H_0(\nabla_\xi H(\xi)) = \left[\e^2+H^2(\xi)  \right]^{\frac{p-2}{2}}H(\xi),
\end{split}
\end{equation*}
thus \eqref{est:lowerae} follows from the equivalence of norms on $\R^n$ and \eqref{eqn:ut4-pippo} follows from \eqref{eqn:ut1}. 

Also, by \eqref{eq:chain} and \eqref{ell:A}, we have that
\begin{equation}\label{stima-grad-a}
\norm{\nabla a_\e(\nabla u_\e)}=\norm{A_\e\,D^2 u_\e}\leq C(n,p,H)[\e^2+H^2(\nabla u_\e)]^{\frac{p-2}{2}}\norm{D^2 u_\e} \quad \text{a.e. on } \Omega,
\end{equation}
therefore \eqref{eqn:ut4} follows immediately from \eqref{eqn:ut4-pippo}.
\end{proof}

\bigskip

Now we estimate the term $\int_{B_R} \abs{a_\e(\nabla u_\e)}^2 dx$, where $B_R$ is any open ball such that $\overline{B_{2R}} \subset \Omega.$ More precisely we have the following result.

\bigskip

\begin{proposition} \label{prop-quasi-stima-finale}
Let $u_\e$ be a solution of \eqref{eqn:ue}. Then, for any $\e\in (0,1)$ and for any open ball 
$B_{2R} \subset  \subset \Omega$ we have 

\begin{equation}\label{eqn:ut1-bis}
\int_{B_R} \vert a_\e(\nabla u_\e) \vert^{2} dx \leq C_3 \Bigl[ R^{-n} \left( \int_{B_{2R}\setminus B_R}\vert a_\e(\nabla u_\e) \vert dx \right)^2+ R^2\int_{B_{2R}} f_\e^2 dx \Bigr]
\end{equation}
\begin{equation}\label{eqn:ut1-tris}
\int_{B_\frac{R}{2}} \Vert \nabla a_\e(\nabla u_\e) \Vert^{2} dx \leq C_4 \Bigl[ R^{-n-2} \left( \int_{B_{2R}\setminus B_R}\vert a_\e(\nabla u_\e) \vert dx \right)^2+ \int_{B_{2R}} f_\e^2 dx \Bigr]
\end{equation}
where $C_3,C_4$ are  constants depending only on $n,p$ and $H$. 
\end{proposition}

\begin{proof}
Thanks to Lemma \ref{lemma_reg_u_ep} we have $\eta a_\e^k(\nabla u_\e) \in H^1_c(\Omega)$ for any $ k=1,..., n$ and for $\eta \in C^{0,1}_c(\Omega)$ whose support is contained in $\overline{B_{2R}} \subset \Omega$.

We first consider the case $n\geq 3$. 

\textit{Case $n\geq 3$.} Since
\begin{equation}\label{e:stima1-n3}
\begin{split}
&\int_{\Omega} \vert \eta\,a_\e(\nabla u_\e) \vert^{2^*}dx = \int_{\Omega} \left( \vert \eta\,a_\e(\nabla u_\e) \vert^2\right)^{\frac{2^*}{2}}dx 
\\
&= \int_{\Omega} \left( \sum_{k=1}^{n} \vert \eta\,a_\e^k(\nabla u_\e) \vert^2\right)^{\frac{2^*}{2}}dx \leq C(n) \int_{\Omega} \sum_{k=1}^{n} \vert \eta\,a_\e^k(\nabla u_\e) \vert^{2^*} dx \,,
\end{split}
\end{equation}
then the Sobolev embedding $H^1(\Omega)\hookrightarrow L^{2^*}(\Omega)$ 
yields
\begin{equation}\label{e:sob1}
\begin{split}
&\int_{\Omega} \vert \eta\,a_\e(\nabla u_\e) \vert^{2^*}dx \leq C'(n) \left[ \sum_{k=1}^{n} \left(  \int_{\Omega} \vert \nabla(\eta\,a_\e^k(\nabla u_\e)) \vert^{2} dx\right)^{\frac{2^*}{2}} \right]
\\
& \leq C''(n) \sum_{k=1}^{n} \left[ \int_{\Omega} \left( \eta^2  \vert \nabla a_\e^k(\nabla u_\e)) \vert^{2} + \vert a_\e^k(\nabla u_\e)) \vert^2 \vert \nabla \eta  \vert^2 \right) dx \right]^{\frac{2^*}{2}}
\\
& \leq nC''(n)\left[ \int_{\Omega} \left( \eta^2  \Vert \nabla a_\e(\nabla u_\e)) \Vert^{2} + \vert a_\e(\nabla u_\e)) \vert^2 \vert \nabla \eta  \vert^2  \right) dx\right]^{\frac{2^*}{2}}  \,.
\end{split}
\end{equation}
Now we use \eqref{eqn:ut4} in the latter to infer
\begin{equation}\label{e:stima2-n3}
\begin{split}
& \int_{\Omega} \vert \eta\,a_\e(\nabla u_\e) \vert^{2^*}dx \leq 
nC''(n) \left[ (C_2+1) \int_{\Omega}\abs{a_\e(\nabla u_\e)}^2\abs{\nabla \eta}^2 dx+C_2 \int_{\Omega}\eta^2 f_\e^2  dx\right]^{\frac{2^*}{2}}
\\
& \leq C(n,p,H) \left[ \int_{\Omega} \abs{a_\e(\nabla u_\e)}^2\abs{\nabla \eta}^2 dx+ \int_{\Omega} \eta^2 f_\e^2 dx\right]^{\frac{2^*}{2}}
\\
& \leq C'(n,p,H) \left[  \left( \int_{\Omega} \abs{a_\e(\nabla u_\e)}^2\abs{\nabla \eta}^2 \right)^{\frac{2^*}{2}} dx+ \left(\int_{\Omega} \eta^2 f_\e^2 dx  \right)^{\frac{2^*}{2}}\right] \,.
\end{split}
\end{equation}
Let $R \leq t<s\leq 2R$ and let $\eta=\eta_{t,s}\in C^{0,1}_c( \Omega)$ be a cut off function with $0\leq \eta\leq 1$ and such that
\begin{equation}\label{test-function-2R}
\eta\equiv 1\quad\text{on }B_t,\quad \eta=0 \quad\text{on } \Omega \setminus B_s, \quad \abs{\nabla \eta}\leq\frac{1}{s-t} \quad \text{on} \quad \Omega.
\end{equation}
Then from \eqref{e:stima2-n3} we have
\begin{equation}\label{e:sob2}
\int_{B_t} \vert a_\e(\nabla u_\e) \vert^{2^*} dx \leq C''(n,p,H) \left[\frac{1}{(s-t)^{2^*}} \left( \int_{B_s \setminus B_R} \vert a_\e(\nabla u_\e) \vert^{2} dx\right)^{\frac{2^*}{2}} dx + \left(\int_{B_{2R}} f_\e^2 dx \right)^{\frac{2^*}{2}}  \right].
\end{equation}

Following \cite[Remark 6.12]{giu}, let $r=2^*/2>1$ and consider $\sigma \in(0,1)$. Let 
\begin{equation*}
\a=\left(\frac{1-\sigma}{r-\sigma}\right)r \in(0,1)
\end{equation*}
so that
$$
\frac{r}{\a}=\frac{r-\sigma}{1-\sigma}>1,\quad \left(\frac{r}{\a} \right)'=\frac{r-\sigma}{r-1}\quad\text{and} \quad (1-\a)\left(\frac{r}{\a} \right)'=\sigma.
$$
By Holder's inequality we have
\begin{equation}
\begin{split} 
& \int_{B_s \setminus B_R}\vert a_\e(\nabla u_\e)\vert^{2}dx =\int_{B_s \setminus B_R} \vert a_\e(\nabla u_\e) \vert^{2\a} \,\vert a_\e(\nabla u_\e) \vert^{2(1-\a)} dx 
\\
&\leq 
\left( \int_{B_s \setminus B_R} |a_\e(\nabla u_\e) \vert^{2r} dx \right)^{\frac{1-\sigma}{r-\sigma}} 
\left( \int_{B_s \setminus B_R} |a_\e(\nabla u_\e) \vert^{2\sigma} dx \right)^{\frac{r-1}{r-\sigma}} \,.
\end{split}
\end{equation}
Thus, since $-2^*= -2n/(n-2)=n(1-r)$, from \eqref{e:sob2} and the latter we obtain
\begin{equation}\label{e:stima3-n3}
\begin{split}
& \int_{B_t} \vert a_\e(\nabla u_\e) \vert^{2^*} dx 
 \\
 &
\leq C''(n,p,H) (s-t)^{n(1-r)} \left(\int_{B_s \setminus B_R} | a_\e(\nabla u_\e) \vert^{2r} dx \right)^{r\left(\frac{1-\sigma}{r-\sigma}\right)} \left(\int_{B_s \setminus B_R} | a_\e(\nabla u_\e) \vert^{2\sigma} dx \right)^{r\left(\frac{r-1}{r-\sigma}\right)} 
\\
& + C''(n,p,H)  \left( \int_{B_{2R}} f_\e^2 dx \right)^{r} 
\\
& \leq \left(\int_{B_s} | a_\e(\nabla u_\e) \vert^{2r} dx \right)^{r\left(\frac{1-\sigma}{r-\sigma}\right)} 
\Bigl[ C''(n,p,H) (s-t)^{n(1-r)} \left(\int_{B_s \setminus B_R} | a_\e(\nabla u_\e) \vert^{2\sigma} dx \right)^{r\left(\frac{r-1}{r-\sigma}\right)} \Bigr]
\\
& + C''(n,p,H)  \left( \int_{B_{2R}} f_\e^2 dx \right)^{r} 
\end{split}
\end{equation}
and therefore, via weighted Young's inequality with conjugate exponents $\frac{r-\sigma}{r(1-\sigma)}$ and $\frac{r-\sigma}{\sigma(r-1)}$,
we obtain
\begin{equation*}
\begin{split}
 \int_{B_t}  \vert  a_\e(\nabla u_\e) \vert^{2^*} dx  & \leq \frac{1}{2}\int_{B_s}\vert a_\e(\nabla u_\e) \vert^{2r} dx + \tilde{C}(s-t)^{-(r-\sigma)\frac{n}{\sigma}}\left( \int_{B_s \setminus B_R}\vert a_\e(\nabla u_\e) \vert^{2\sigma} dx \right)^{\frac{r}{\sigma}} \\ & \hspace{4em} +C''(n,p,H)\left( \int_{B_{2R}} f_\e^2 dx \right)^{r} \\
 & \leq \tilde{C}(s-t)^{-(r-\sigma)\frac{n}{\sigma}}\left( \int_{B_{2R} \setminus B_R}\vert a_\e(\nabla u_\e) \vert^{2\sigma} dx \right)^{\frac{r}{\sigma}} 
+C''(n,p,H) \left( \int_{B_{2R}} f_\e^2 dx \right)^{r}  \\ & \hspace{4em}  + \frac{1}{2}\int_{B_s}\vert a_\e(\nabla u_\e) \vert^{2^*} dx
\end{split}
\end{equation*}
where $ \tilde{C}$ is a constant depending only on $n, p, \sigma$ and $H$. 

By applying \cite[Lemma 6.1]{giu} with
\[
Z(t)=\int_{B_t}\vert a_\e(\nabla u_\e) \vert^{2^*} dx,
\]
and by choosing $\sigma=\frac{1}{2}$, from the above inequality we obtain 

\begin{equation} \label{aboveineq-bis}
\int_{B_{R}}\vert a_\e(\nabla u_\e) \vert^{2^*} dx \leq C'''\, R^{-(r-\sigma)\frac{n}{\sigma}}\left( \int_{B_{2R} \setminus B_R} \vert a_\e(\nabla u_\e) \vert dx \right)^{2r}+C'''\left( \int_{B_{2R}} f_\e^2 dx \right)^{r} 
\end{equation} 
where $ C'''$ is a constant depending only on $n, p, H$. 

Then Holder inequality and \eqref{aboveineq-bis} imply
\begin{equation}\label{e:stima4-n3}
\begin{split}
\int_{B_R} \vert a_\e(\nabla u_\e) \vert^{2} dx \leq C'''_1 \abs{B_R}^{2/n} \Bigl[ R^{-(r-\sigma)\frac{n-2}{\sigma}}\left( \int_{B_{2R}\setminus B_R}\vert a_\e(\nabla u_\e) \vert dx \right)^2+\int_{B_{2R}} f_\e^2 dx \Bigr]
\end{split}
\end{equation}
where $C'''_1 $ is a constant depending only on $n, p, H$. 

A short computation yields $(r-\sigma)\frac{n-2}{\sigma} = n+2$, therefore the latter gives
\begin{equation} \label{above2}
\begin{split}
\int_{B_R} \vert a_\e(\nabla u_\e) \vert^{2} dx \leq C'''_2 R^2 \Bigl[ R^{-(n+2)} \left( \int_{B_{2R}\setminus B_R}\vert a_\e(\nabla u_\e) \vert dx \right)^2+\int_{B_{2R}} f_\e^2 dx \Bigr]
\end{split}
\end{equation}
where $C'''_2 $ is a constant depending only on $n, p, H$. This proves \eqref{eqn:ut1-bis}.

To prove \eqref{eqn:ut1-tris} we make use of \eqref{eqn:ut4} by letting $\eta\in C^{0,1}_c(\Omega)$ be a cut-off function with $0\leq \eta\leq 1$ and such that
\[
\eta\equiv 1\quad\text{in  }B_{R/2}, \quad \eta=0 \quad\text{on } \Omega \setminus B_R, \quad \abs{\nabla \eta} \leq 2/R \quad\text{on } \Omega, 
\]
which leads to 
\begin{equation}\label{eqn:ut4-bis}
\int_{B_\frac{R}{2}} \norm{\nabla a_\e(\nabla u_\e)}^2 \leq 4C_2 R^{-2}  \int_{B_R}\abs{a_\e(\nabla u_\e)}^2 dx+ C_2 \int_{B_R} f_\e^2 dx
\end{equation}
Inserting \eqref{above2} into the latter 
yields \eqref{eqn:ut1-tris}.

\medskip

\textit{Case $n=2$.} In this case we observe that, for any $\theta > 2$, it holds
\begin{equation}\label{sobolev2D}
\int_{\Omega} \vert \eta a_\e^k(\nabla u_\e)\vert^\theta \leq C(\theta)R^2 
\left( \int_{\Omega} \nabla \vert (\eta a_\e^k(\nabla u_\e) \vert^2 \right)^{\frac{\theta}{2}} \,.
\end{equation}
Here we have used that $\eta a_\e^k(\nabla u_\e) \in H^{1}_c(\Omega)$ and its support is contained in $\overline{B_{2R}} \subset \Omega$ (see for instance \cite[Theorem 12.33]{leo}). Now we repeat the previous computations with any $\theta>2$ fixed. This leads to  \eqref{e:sob2} with $ 2^*$ replaced by $ \theta$ and $C''(n,p,H)$ replaced by $C''(n,p,H,\theta)R^2$, i.e., 
\begin{equation}\label{quasi-giusti-2}
\int_{B_t} \vert a_\e(\nabla u_\e) \vert^{\theta} dx \leq C''(n,p,H,\theta) R^2 \left[\frac{1}{(s-t)^{\theta}} \left( \int_{B_s} \vert a_\e(\nabla u_\e) \vert^{2} dx\right)^{\frac{\theta}{2}} dx + \left(\int_{B_{2R}} f_\e^2 dx \right)^{\frac{\theta}{2}} dx \right].
\end{equation}

Now we choose $r = \frac{\theta}{2} >1$ and we repeat the computations after formula \eqref{e:sob2}. This leads to 
\begin{equation}\label{e:stima3-n2}
\begin{split}
& \int_{B_t} \vert a_\e(\nabla u_\e) \vert^{\theta} dx 
\\
&
\leq C''(n,p,H,\theta) R^2 (s-t)^{-2r} \left(\int_{B_s \setminus B_R} a_\e(\nabla u_\e) \vert^{2r} dx \right)^{r\left(\frac{1-\sigma}{r-\sigma}\right)} \left(\int_{B_s \setminus B_R}a_\e(\nabla u_\e) \vert^{2\sigma} dx \right)^{r\left(\frac{r-1}{r-\sigma}\right)} 
\\
& + C''(n,p,H,\theta) R^2 \left( \int_{B_{2R}} f_\e^2 dx \right)^{r} 
\leq \tilde{C}R^{\frac{2(r-\sigma)}{\sigma(r-1)}} (s-t)^{-\frac{2r(r-\sigma)}{\sigma(r-1)}}\left( \int_{B_{2R} \setminus B_R}\vert a_\e(\nabla u_\e) \vert^{2\sigma} dx \right)^{\frac{r}{\sigma}} 
\\
&+C''(n,p,H,\theta) R^2 \left( \int_{B_{2R}} f_\e^2 dx \right)^{r} + \frac{1}{2}\int_{B_s}\vert a_\e(\nabla u_\e) \vert^{\theta} dx
\end{split}
\end{equation}
where $ \tilde{C}$ is a constant depending only on $n, p, \sigma, H$ and $\theta$. 

By by choosing $\sigma=\frac{1}{2}$ and applying \cite[Lemma 6.1]{giu}
we obtain 

\begin{equation} \label{aboveineq-bis-n2}
\begin{split}
&\int_{B_{R}}\vert a_\e(\nabla u_\e) \vert^{\theta} dx \leq C'''\,
R^{\frac{2(r-\sigma)}{\sigma(r-1)}} R^{-\frac{2r(r-\sigma)}{\sigma(r-1)}}\left( \int_{B_{2R} \setminus B_R}\vert a_\e(\nabla u_\e) \vert^{2\sigma} dx \right)^{\frac{r}{\sigma}} + C''' R^2\left( \int_{B_{2R}} f_\e^2 dx \right)^{r} 
\\
&= C'''\, R^{-2(\theta -1)} \left( \int_{B_{2R} \setminus B_R}\vert a_\e(\nabla u_\e) \vert dx \right)^{\theta} + C''' R^2\left( \int_{B_{2R}} f_\e^2 dx \right)^{\frac{\theta}{2}}
\end{split}
\end{equation} 
where $ C'''$ is a constant depending only on $n, p, H$ and $\theta$.

Then Holder inequality and \eqref{aboveineq-bis-n2} imply
\begin{equation}\label{e:stima4-n3}
\begin{split}
\int_{B_R} \vert a_\e(\nabla u_\e) \vert^{2} dx \leq C'''_1 \Bigl[ R^{-2} \left( \int_{B_{2R}\setminus B_R}\vert a_\e(\nabla u_\e) \vert dx \right)^2 + R^2 \int_{B_{2R}} f_\e^2 dx \Bigr]
\end{split}
\end{equation}
where $C'''_1 $ is a constant depending only on $n, p, H$ and $\theta$. Then \eqref{eqn:ut1-bis} follows by fixing a value of $ \theta>2$. From the latter it is immediate to infer \eqref{eqn:ut1-tris}.


\end{proof}


\section {Proof of the main results} \label{sect_proofs}
In this section we prove the main results of this paper. 

\begin{proof}[Proof of Theorem \ref{thm_main1}]
It suffices to apply the estimates we found in the previous sections for the approximating sequence $u_\e$, and then pass to the limit as $\e\to0$.

Let us fix $ \Omega' \subset\subset\Omega$ and consider $u_\e$ solutions to \eqref{eqn:ue}.
From \eqref{def-a-epsilon} we have 
$$
\abs{a_\e(\nabla u_\e)}\leq C(H) \left(\e^2+H^2(\nabla u_\e)\right)^{\frac{p-1}{2}} \,,
$$
and therefore \eqref{eqn:stime1} yields
\[
\norm{a_\e(\nabla u_\e)}_{L^1({\Omega'})}\leq C \,,
\]
where $C$ does not depend on $\e$. Then from Proposition \eqref{prop-quasi-stima-finale} and a standard covering argument we infer that 
\begin{equation}
\norm{\nabla a_\e(\nabla u_\e)}_{H^1(\Omega')}\leq C,
\end{equation}
where $C$ does not depend on $\e$. 

Since those estimates are uniform in $\e$, we can extract a subsequence, relabelled as $u_\e$, such that
\begin{equation}\label{conv-debole-puntuale}
\begin{split}
a_\e(\nabla u_\e)\to {h}\quad \text{weakly in } H^1_{loc}(\Omega), \quad  \text{strongly in } L^2_{loc}(\Omega)\quad \text{and }  \quad \text{a.e. in} \quad \Omega,
\end{split}
\end{equation}
for some $h \in H^1_{loc}(\Omega)$.

From the $L^p$ convergence $\nabla u_\e\to \nabla u$, we have (up to a subsequence, still denoted by $u_\e$)
\[
\nabla u_\e \to\nabla u \quad \text{a.e. in} \quad \Omega.
\]
Hence 
\[
a_\e(\nabla u_\e)\to a(\nabla u)\quad \text{a.e. in} \quad \Omega,
\]
and so $h = a(\nabla u)$ thanks to \eqref{conv-debole-puntuale}.

Estimates \eqref{est:nablaanyp} and \eqref{est:anyp} then follows by letting $\e\to0$ in Proposition \ref{prop-quasi-stima-finale}. Finally, the estimate \eqref{est:anypL1} follows immediately from \eqref{def-a}. 
\end{proof}

\medskip

As already observed in Remark \ref{rem-thm_D2u} and Remark \ref{rem-thm_main2}, Theorem \ref{thm_D2u} and Theorem  \ref{thm_main2} are special cases of two more general results that we state and prove hereafter.  To this end we first introduce the assumptions on the source term $f$ :

\begin{equation}\label{ipotesi-f-campa} 
\begin{cases}
\,\, \text{if} \quad p > \frac{n}{2} \qquad \exists \, \lambda \in (n-2,n) \qquad \quad \quad \quad : \quad 
f \in  {\mathcal M}^{2, \lambda}_{loc}(\Omega),
\\
\\
\,\, \text{if} \quad p \leq \frac{n}{2} \qquad \exists \, \lambda \in (n-2,n), \,\, \exists \, s > \frac{n}{p} \quad :  \quad f \in L^s_{loc}(\Omega) \cap {\mathcal M}^{2, \lambda}_{loc}(\Omega), 
\end{cases}
\end{equation}
where we have denoted by ${\mathcal M}^{2, \lambda}$ the classical Morrey space.
Then we have 

\smallskip

\begin{remark} \label{caso-particolare}

\item  i) If $f$ satisfies \eqref{ipotesi-f-campa}, then $f \in L^q_{loc}(\Omega)$ where $q$ fulfills \eqref{def-q}, and therefore Theorem  \ref{thm_main1} applies. 

\item ii) If $f \in L^r_{loc}(\Omega)$, $r>n$, then $f$ satisfies \eqref{ipotesi-f-campa}. Indeed, by Holder inequality,  we have that $f \in{\mathcal M}^{2, n-\frac{2n}{r}}_{loc}(\Omega)$ (and $n-\frac{2n}{r}  \in (n-2,n)$, since $r>n$).  Moreover, $ \norm{f}_{{\mathcal M}^{2, n-\frac{2n}{r}}(\Omega')} \leq
C(n,r) \norm{f}_{L^r (\Omega')}$ for any open subset $\Omega'  \subset \subset \Omega$. 
\textit{Therefore, Theorem \ref{thm_D2u} and Theorem  \ref{thm_main2} are special cases of the two following general results.} 
\end{remark}

\medskip

\begin{theorem} \label{thm_D2u-generale}
	Assume $1<p\leq 2$ and let $u\in W^{1,p}_{loc}(\Omega)$ be a local weak solution of \eqref{eqn1} where $H$ satisfies \eqref{eqn:ellH} and $f$ satisfies  \eqref{ipotesi-f-campa}.
	Then
	\[
	u\in H^2_{loc}(\Omega) \cap C^{1,\beta}_{loc}(\Omega)
	\]
	for some $\beta \in (0,1)$ depending only on $n,p, \lambda$ and $H$. 
	
	Moreover, for any open ball $B_{2R} \subset \subset \Omega$ we have 
	\[
	\int_{B_{R/2}}\norm{D^2 u}^2 dx\leq C \Bigl[ R^{-n-2} \norm{a(\nabla u)}_{L^1(B_{2R} \setminus B_R)}^2 + \norm{f}_{L^2(B_{2R} )}^2 \Bigr],
	\]
	where  $C$ is a constant depending on $p,n,H, \lambda, B_R, B_{2R}, \norm{u}_{W^{1,p}(B_{2R})}, \norm{f}_{L^{\max\{2,s\}}(B_{2R})}$ and $  \norm{f}_{{\mathcal M}^{2, \lambda}(B_{2R})}.$
	
	In particular, when $p=2$ we have 
	\[
	\int_{B_{R/2}}\norm{D^2 u}^2 dx\leq C \Bigl[ R^{-n-2} \norm{a(\nabla u)}_{L^1(B_{2R} \setminus B_R)}^2 + \norm{f}_{L^2(B_{2R} )}^2 \Bigr],
	\]
	where  $C$ is a constant depending only on $n,H$.
\end{theorem}

\medskip

\begin{theorem} \label{thm_main2-generale}
	Let $u\in W^{1,p}_{loc}(\Omega)$ be a local solution of \eqref{eqn1}, where $H$ satisfies \eqref{eqn:ellH} and $f$ satisfies  \eqref{ipotesi-f-campa}. Then
	\[
	u\in C^{1,\beta}_{loc}(\Omega)
	\]
	for some $\beta \in (0,1)$ depending only on $n,p, \lambda$ and $H$. 
	
	Moreover, for any open ball $B_{2R} \subset \subset \Omega$ we have 
	\begin{equation}\label{int:est}
	\int_{B_{R/2} \setminus Z}\left[H^2(\nabla u)\right]^{p-2}\norm{D^2 u}^2 dx\leq C,
	\end{equation}
	where $Z$ denotes the set of critical points of $u$ and 
	$C$ is a constant depending on $p,n,H, $
	$\lambda, B_R, B_{2R}, \norm{u}_{W^{1,p}(B_{2R})}, \norm{f}_{L^{\max\{2,s\}}(B_{2R})}$ and $  \norm{f}_{{\mathcal M}^{2, \lambda}(B_{2R})}.$
\end{theorem}


\medskip

To prove  Theorem \ref{thm_D2u-generale} and Theorem  \ref{thm_main2-generale} we need the following useful auxiliary result (inspired by the reading of Section 5 of \cite{lieb}). 

\medskip

\begin{lemma}\label{campa} 
	Assume $n \geq2$ and let $U$ be an open bounded set of $\R^n$ of class $C^2$. Let $f$ be a function belonging to the Morrey space 
	${\mathcal M}^{2, \lambda}(U)$ with $ n-2 <  \lambda <n$ and set $ \alpha =  \frac{\lambda -n +2}{2} \in (0,1)$.
	Then there exists $F \in H^1(U, \R^n) \cap C^{0, \alpha}_{loc}(U, \R^n)$ such that 
	\begin{equation}\label{eqn:campa}
	-\mathrm{div}F = f \quad in \quad U
	\end{equation}
	and, for any open Lipschitz set $U' \subset \subset U$, 
	\begin{equation} \label{stima-campa}
	\norm{F}_{C^{0, \alpha}(U')} \leq C 
    \norm{f}_{{\mathcal M}^{2, \lambda}(U)},
	\end{equation}
	where $C$ is a constant depending only on $n, \lambda, U'$ and $U$. 
\end{lemma}

\begin{proof} The proof relies on some results of Campanato and Morrey. \footnote{Recall that the Morrey space ${\mathcal M}^{2, \lambda}(A)$ is isomorphic (as Banach space) to the Campanato space ${\mathcal L}^{2, \lambda}(A)$ whenever $A$ is an open  bounded Lipschitz set of $\R^n$ and $ 0 \leq \lambda< n$. We shall freely use this result in the course of the proof. More details on this property as well as other useful results used in this paper on Morrey's and Campanato's spaces can be found in \cite{giu}[Section 2.3]}

Let $u \in H^1_0(U) \cap H^2(U)$ be the unique weak solution to $ - \Delta u =f$ in $U$ and recall that $\norm{u}_{H^2(U)}^2 \leq C_1 \norm{f}_{L^2(U)}^2,$ for some constant $C_1$ depending only on $ n$ and $U$. Also, by a result of Campanato \cite[Teorema 10.I]{Campa3} we know that 
	\begin{equation} \label{stima-campa-loc}
    \norm{\partial_{jk}u}_{{\mathcal M}^{2, \lambda}(U')}^2 \leq C_2 \left[ \norm{u}_{H^2(U)}^2 + \norm{f}_{{\mathcal M}^{2, \lambda}(U)}^2\right] \qquad  \forall
     \, j,k=1,...,n
    \end{equation}
    where the constant $C_2$ depends only on $ \lambda, n$ and $U'$.  Hence,
\begin{equation} \label{stima-campa-loc2}
\norm{\nabla \partial_k u}_{{\mathcal M}^{2, \lambda}(U')}^2 
\leq C_3 \norm{f}_{{\mathcal M}^{2, \lambda}(U)}^2 \qquad  \forall
\, j,k=1,...,n
\end{equation}   
where $C_3$ is a constant that depends only on $ \lambda, n, U'$ and $U$.
Set $ w = \partial_k u$, then Poincar\'e inequality and \eqref{stima-campa-loc2} imply that, for any $ x_0 \in U'$ and any $0< \rho < \frac{ dist(U', \, \partial U)}{2}$, 
\begin{equation} \label{stima-campa-loc3}
\int_{B_\rho(x_0)}  \vert w - w_{B_\rho(x_0)} \vert^2 dx \leq c \rho^2 \int_{B_\rho(x_0)}  \vert \nabla  \partial_k u \vert^2 \leq c \rho^2 C_3  \norm{f}_{{\mathcal M}^{2, \lambda}(U)}^2 \rho^{\lambda} = c C_3  \norm{f}_{{\mathcal M}^{2, \lambda}(U)}^2\rho^{\lambda +2}
\end{equation} 
where $ w_{\omega}:= \frac{1}{ \vert \omega \vert} \int_{\omega} w \, dx$ and $c=c(n)$.  

Moreover, when $\rho \geq \frac{ dist(U', \, \partial U)}{2}$, we have 
\begin{equation} \label{stima-campa-loc3bis}
\begin{split}
&\int_{U \cap  B_\rho(x_0)}  \vert w - w_{U \cap  B_\rho(x_0)}  \vert^2 dx \leq 
2 \norm{w}_{L^2(U)}^2 \leq 2 \norm{w}_{L^2(U)}^2 \left[ \frac{2\rho}{dist(U', \, \partial U)}\right]^{\lambda +2}  
\\
&= 2 \left[ \frac{2}{dist(U', \, \partial U)}\right]^{\lambda +2}
\norm{\partial_k u}_{L^2(U)}^2 \rho^{\lambda +2} \leq 2 \left[ \frac{2}{dist(U', \, \partial U)}\right]^{\lambda +2}
C_1^2\norm{f}_{L^2(U)}^2 \rho^{\lambda +2} 
\end{split} 
\end{equation} 
Combining \eqref{stima-campa-loc3} and \eqref{stima-campa-loc3bis} we immediately get that $ \partial_k u$ belongs to the Campanato space 
${\mathcal L}^{2, \lambda+2}(U')$ and
\begin{equation} \label{stima-campa-loc4}
\norm{\partial_k u}_{{\mathcal L}^{2, \lambda+2}(U')}^2 
\leq C_4 \norm{f}_{{\mathcal M}^{2, \lambda}(U)}^2 \qquad  \forall
\, k=1,...,n
\end{equation}   
where $C_4$ is a constant depending only on $n, \lambda, U'$ and $U$. 

Now, since $ n <  \lambda+2 < n +2$, the well-known integral characterisation of Holder spaces by Campanato \cite {Campa1, Campa2, giu} tell us that 

\begin{equation} \label{stima-campa-loc5}
\partial_k u \in C^{0,\frac{\lambda -n +2}{2}}(\overline {U'}), \quad 
\norm{\partial_k u}_{C^{0,\frac{\lambda -n +2}{2}}(\overline {U'})}
\leq C_5 \norm{f}_{{\mathcal M}^{2, \lambda}(U)} \qquad  \forall
\, k=1,...,n
\end{equation} 
where $C_5$ is a constant depending only on $n, \lambda, U'$ and $U$.  

The desired conclusion then follows by taking $ F =  \nabla u$.
\end{proof} 

\medskip


We are now ready to prove Theorem \ref{thm_D2u-generale}.

\medskip

\begin{proof}[Proof of Theorem \ref{thm_D2u-generale}]
Set 
\begin{equation}
\tilde s:= \begin{cases}
2  \qquad \text{if} \quad p > \frac{n}{2},  
\\
s  \qquad \text{if} \quad p  \leq \frac{n}{2}.
\end{cases}
\end{equation}
Let us consider an open ball 
	$B_{2R} \subset \subset \Omega$  and let $f_\e$ and $u_\e$  be as in Section \ref{sect_approximation}. 
Recall that, in the course of the proof of Theorem \ref{thm_main1}, we proved that
\begin{equation}\label{stima-sob1}
\norm{u_\e}_{W^{1,p}(B_{2R})} \leq C_1' := C_1'(p,n,H, B_{2R}, \norm{u}_{W^{1,p}(B_{2R})}, \norm{f}_{L^ {\tilde s}(B_{2R})})  
\end{equation} 
\begin{equation}\label{stima-sob1-bis}
\norm{a_\e(\nabla u_\e)}_{L^1(B_{2R})} \leq C_1'
\end{equation}
and that, up to a subsequence, 
\begin{equation}\label{conv-debole-bis-bis}
\begin{split}
\nabla u_\e \to \nabla u \quad \text{strongly in } W^{1,p}_{loc}(\Omega) \quad  \text{and }  \quad \text{a.e. in} \quad \Omega,
\end{split}
\end{equation}	
\begin{equation}\label{conv-debole-puntuale-bis}
\begin{split}
a_\e(\nabla u_\e)\to {a(\nabla u)}\quad \text{weakly in } H^1_{loc}(\Omega), \quad  \text{strongly in } L^2_{loc}(\Omega)\quad \text{and }  \quad \text{a.e. in} \quad \Omega.
\end{split}
\end{equation}	
By making use of \eqref{stima-sob1} we have $u_\e \in C^0(\Omega)$ and the following bound
\begin{equation}\label{stima-james1}
\norm{u_\e}_{L^\infty(B_R)} \leq C_2' := C_2'(p,n,H, B_{2R}, \norm{u}_{W^{1,p}(B_{2R})}, \norm{f}_{L^ {\tilde s}(B_{2R})}).
\end{equation}
Indeed, if $ p>n$ we have  $\norm{u_\e}_{L^\infty(B_R)} \leq C(B_R,p) \norm{u_\e}_{W^{1,p}(B_{R})}$ by Sobolev embedding, and so \eqref{stima-james1}
follows from \eqref{stima-sob1}. When $p \leq n$ we have 
$\norm{u_\e}_{L^\infty(B_R)} \leq C'(p,n,H, B_{2R}, \norm{u_\e}_{L^p(B_{2R})}, \norm{f}_{L^ {\tilde s}(B_{2R})})$, by the celebrated results in \cite{ser}, and once again \eqref{stima-james1} follows from \eqref{stima-sob1}.

Now we observe that $f_\e \in {\mathcal M}^{2, \lambda}(B_{2R})$ and $
\norm{f_\e}_{{\mathcal M}^{2, \gamma}(B_{2R})} \leq \norm{f}_{{\mathcal M}^{2, \gamma}(B_{2R})}$. We can therefore use Lemma \ref{campa} to obtain vector fields $F_\e \in C^{0, \alpha}(\overline{B_R})$ such that 
\begin{equation} \label{stima-campa-campa}
\norm{F_\e}_{C^{0, \alpha}(B_R)} \leq C 
\norm{f_\e}_{{\mathcal M}^{2, \lambda}(B_{2R})} \leq C 
\norm{f}_{{\mathcal M}^{2, \lambda}(B_{2R})} 
\end{equation}
where $\alpha =  \frac{\lambda -n +2}{2} \in (0,1)$ and $C$ is a constant depending only on $n, \lambda, B_R$ and $B_{2R}$. 

Now we set $A_\e(x, \xi) := a_\e (\xi) - F_\e(x)$, $ (x,\xi) \in B_R \times (\R^n \setminus \{0\})$ and observe that 
 \begin{equation} \label{eq-campa-liebermann}
-\mathrm{div}(A_\e(x, \nabla u_\e)) =0 \quad\text{in } B_R.
 \end{equation}
 We can therefore apply \cite[Theorem 1.7]{lieb2} to obtain $\beta = \beta(n,p,H, \lambda) \in (0,1)$  such that 
 \begin{equation} \label{stima-campa-liebermann}
	\norm{u_\e}_{C^{1,\beta}(B_\frac{R}{2})}\leq C_3'=C_3'(p,n,H, \lambda, B_R, B_{2R}, \norm{u}_{W^{1,p}(B_{2R})}, \norm{f}_{L^ {\tilde s}(B_{2R})}, \norm{f}_{{\mathcal M}^{2, \lambda}(B_{2R})}).
 \end{equation} 
Hence, up to a subsequence, $u_\e\to u$ in $C^1_{loc}(\Omega)$, $ u \in  C^{1,\beta}_{loc}(\Omega)$. 

By \eqref{eqn:ut4-pippo} and $p\leq 2$ we get
\begin{equation}\label{dis:hessiana}
\begin{split}
&\int_{B_\frac{R}{2}}\norm{D^2 u_\e}^2 dx \leq
C_4' 
\int_{B_\frac{R}{2}}\left[ \e^2+H^2(\nabla u_\e) \right]^{p-2}\norm{D^2 u_\e}^2 dx \leq 
\\
& C_4' C_2 \left[ \frac{4}{R^2}\int_{B_R}\abs{a_\e(\nabla u_\e)}^2 dx+ \int_{B_R} f_\e^2 dx\right]
\end{split}
\end{equation}
where $ C_2$ is a constant depending only on $n,p, H$ and $C_4'$ is a positive constant depending only on $C_3'$ (note that one can take $C_4' =1$ when $p=2$). 
Then, inserting \eqref{eqn:ut1-bis} into the latter yields
\begin{equation}
\begin{split}
&\int_{B_\frac{R}{2}}\norm{D^2 u_\e}^2 dx \leq
C_4' C_2 \left[ \frac{4}{R^2}\int_{B_R}\abs{a_\e(\nabla u_\e)}^2 dx+ \int_{B_R} f_\e^2 dx\right] 
\\
& \leq C_4' C(n,p,H) \Bigl[ R^{-n-2} \left( \int_{B_{2R}\setminus B_R}\vert a_\e(\nabla u_\e) \vert dx \right)^2+ \int_{B_{2R}} f^2 dx \Bigr]
\\
& \leq C_5' = C_5'(p,n,H, \lambda, B_R, B_{2R}, \norm{u}_{W^{1,p}(B_{2R})}, \norm{f}_{L^ {\tilde s}(B_{2R})}, \norm{f}_{{\mathcal M}^{2, \lambda}(B_{2R})})
\end{split}
\end{equation}
where in the last inequality we have used \eqref{stima-sob1-bis}. Therefore, up to a subsequence, $u_\e\to u$ weakly in $H^2_{loc}(\Omega)$ and the thesis follows by letting $ \e\to0$ in \eqref{dis:hessiana} and then recalling \eqref{est:anyp}.

\end{proof}

\medskip

\begin{proof}[Proof of Theorem \ref{thm_main2-generale}]
We repeat the proof of Theorem \ref{thm_D2u} until the estimate \eqref{stima-campa-liebermann}. Hence, up to a subsequence, 
\begin{equation}\label{miservepoi}
u_\e \to u \quad \text{in} \quad C^1_{loc}(\Omega),  \qquad u \in  C^{1,\beta}_{loc}(\Omega).
\end{equation}
By \eqref{eqn:ut4-pippo}, \eqref{eqn:ut1-bis} and \eqref{stima-sob1-bis} we have that 
\begin{equation}\label{miservepoi-bis}
\begin{split}
& \int_{B_\frac{R}{2}}\left[ \e^2+H^2(\nabla u_\e) \right]^{p-2}\norm{D^2 u_\e}^2 dx \leq C_2 \left[ \frac{4}{R^2}\int_{B_R}\abs{a_\e(\nabla u_\e)}^2 dx+ \int_{B_R} f_\e^2 dx\right] 
\\
& \leq C_5' = C_5'(p,n,H, \lambda, B_R, B_{2R}, \norm{u}_{W^{1,p}(B_{2R})}, \norm{f}_{L^ {\tilde s}(B_{2R})}, \norm{f}_{{\mathcal M}^{2, \lambda}(B_{2R})}) \,;
\end{split}
\end{equation}
therefore, for every $i,j\in\{1,\dots,n\}$ ,
	\begin{equation}
	\phi^{i,j}_\e\coloneqq \left(\e^2+\abs{\nabla u_\e}^2\right)^{\frac{p-2}{2}}\partial_{ij}u_\e
	\end{equation}
	is uniformly bounded in $L^2_{loc}(\Omega)$ w.r.t. $\e>0$. Hence, up to a subsequence,
	\begin{equation}\label{conv:w}
	\phi^{i,j}_\e\to \phi^{i,j} \quad \text{weakly in } L^2_{loc}(\Omega) \quad \text{as } \e\to 0.
	\end{equation}
	
	In view of \eqref{miservepoi-bis}, \eqref{conv:w} and the weak lower semicontinuity of the $L^2$ norm, to get our thesis it is enough to prove that
\begin{equation}\label{identita-limite}
	\phi^{i,j}=\abs{\nabla u}^{p-2}\partial_{ij}u\quad\text{a.e. in } \Omega\setminus Z.
\end{equation}
	
	To this end, we fix an arbitrary open ball $ \mathcal{B}_{2R} \subset\subset \Omega \setminus Z$, then $\abs{\nabla u}\geq 2c>0$ in $\mathcal{B}_{2R}$ by definition of $Z$. Hence, by \eqref{miservepoi}, we have 
	
	\begin{equation}\label{est:belownablaue}
	\abs{\nabla u_\e}\geq c\quad\text{in }\mathcal{B}_{2R}, \quad \text{for all small enough } \e.
	\end{equation}
By using \eqref{miservepoi-bis}, \eqref{est:belownablaue} and \eqref{stima-campa-liebermann} we find
	\begin{equation*}
	\begin{split}
	&\int_{\mathcal{B}_\frac{R}{2}}\norm{D^2 u_\e}dx \leq C(c,p,H, C_3') \int_{\mathcal{B}_\frac{R}{2}} \left[ \e^2+H^2(\nabla u_\e)\right]^{p-2}\norm{D^2 u_\e}^2 dx
	\\
	& \leq C_6' = C_6'(c, p,n,H, \lambda, B_R, B_{2R}, \norm{u}_{W^{1,p}(B_{2R})}, \norm{f}_{L^ {\tilde s}(B_{2R})}, \norm{f}_{{\mathcal M}^{2, \lambda}(B_{2R})}),
	\end{split}
	\end{equation*}
	which implies that $u_\e$ is uniformly bounded in $H^2_{loc}( \Omega \setminus Z)$ and then, up to a subsequence, $u_\e\to u$ weakly in $H^2_{loc}( \Omega \setminus Z)$.
	The latter and \eqref{miservepoi} yield
	\[
	\phi_\e^{i,j}=\left(\e^2+\abs{\nabla u_\e}^2\right)^{\frac{p-2}{2}}\partial_{ij}u_\e\to \abs{\nabla u}^{p-2}\partial_{i,j} u,
	\]
	weakly in $L^2_{loc}( \Omega \setminus Z)$, which proves \eqref{identita-limite} and concludes the proof.
\end{proof}

\begin{proof}[Proof of Proposition \ref{prop:1}]
From Theorem \ref{thm_main1} we know that 
\[
\vert a(\nabla u) \vert \in H^1_{loc}(\Omega).
\]
Thanks to a well-known result due to Stampacchia \cite{stampacchia} we infer that
\[
\frac{\vert a(\nabla u) \vert}{\e+\vert a(\nabla u) \vert} \in H^1_{loc}(\Omega)
\]
for any $\e>0$. Therefore, for any $\vphi \in C^{\infty}_c(\Omega)$, we can use 
$$
\frac{\vert a(\nabla u) \vert}{\e+\vert a(\nabla u) \vert}\vphi
$$ 
as a test function in \eqref{def-eqn-studiata-forma-integrale} and we have
\begin{equation}\label{a:1}
\int_\Omega \frac{\vert a(\nabla u) \vert}{\e+\vert a(\nabla u) \vert}\vphi\,f\,dx
=\int_\Omega \frac{\vert a(\nabla u) \vert}{\e+\vert a(\nabla u) \vert} a(\nabla u)\cdot\nabla\vphi\,dx
 +\e \int_\Omega \, \frac{a(\nabla u ) \cdot \nabla (\vert a(\nabla u) \vert)}{(\e+ \vert a(\nabla u) \vert)^2}  \vphi \,dx.
\end{equation}
We first notice that 
\begin{equation} \label{final}
\int_\Omega \frac{\vert a(\nabla u) \vert}{\e+\vert a(\nabla u) \vert}\vphi\,f\,dx = \int_{\Omega \setminus \{\nabla u=0\}} \frac{\vert a(\nabla u) \vert}{\e+\vert a(\nabla u) \vert}\vphi\,f\, dx.
\end{equation}
Moreover we have
$$
\Bigg{|}\e \frac{a(\nabla u ) \cdot \nabla (\vert a(\nabla u) \vert)}{(\e+ \vert a(\nabla u) \vert)^2} \vphi \Bigg{|} \leq   \nabla (\vert a(\nabla u) \vert) \vert \vphi \vert
$$
where the latter function belongs to $L^1(\Omega)$, independently on $\e$. This implies that we can use the dominated convergence theorem in \eqref{a:1} as $\e\to 0^+ $ and, from \eqref{final}, we obtain
\begin{equation*}
\int_{\Omega\setminus\{\nabla u=0\}} \vphi\,f\,dx=\int_{\Omega} a(\nabla u) \cdot\nabla\vphi\,dx=\int_\Omega\vphi\,f\,dx \,,
\end{equation*}
where in the last equality we used again the equation. Since $\vphi$ is any function in $C^\infty_c(\Omega)$, we get the desired conclusion.
\end{proof}

\begin{proof}[Proof of Corollary \ref{corollary_final}]
This corollary is a straightforward consequence of Proposition \ref{prop:1}. Indeed, the singular set $\{\nabla u = 0\}$ is contained into the set $\{f=0\}$ up to a set of measure zero. Since $|\{f=0\}|=0$ then $|\{\nabla u = 0\}|$. 
\end{proof}

\end{document}